\documentclass[11pt,a4paper]{amsart}
\usepackage{amssymb,amsmath}
\usepackage{graphicx}
\usepackage{subcaption}
\usepackage{comment}
\usepackage{tikz}
\usetikzlibrary{arrows}
\usepackage{xcolor}
\def\({\left(}
\def\){\right)}
\def\Nx{\nabla_x}
\def\Cal{\mathcal}

\def\eb{\varepsilon}

\def\tilde{\widetilde}
\newcommand{\be}{\begin{equation} }
\newcommand{\ee}{\end{equation} }

\def \and{\qquad\text{and}\qquad}

\def\Bbb{\mathbb}
\def\Dt{\partial_t}
\def\Dx{\Delta_x}

\def\({\left(}
\def\){\right)}
\def\Nx{\nabla_x}
\def\divv{\operatorname{div}}
\def\eb{\varepsilon}
\def\Cal{\mathcal}
\def\eb{\varepsilon}

\def\R {\mathbb R}

\def\<{\left<}
\def\>{\right>}

\def \and{\qquad\text{and}\qquad}

\def\Bbb{\mathbb}
\def\Dt{\partial_t}
\def\Dx{\Delta_x}
\def\R {\mathbb R}

\newtheorem{proposition}{Proposition}[section]
\newtheorem{theorem}[proposition]{Theorem}
\newtheorem{corollary}[proposition]{Corollary}
\newtheorem{lemma}[proposition]{Lemma}
\theoremstyle{definition}
\newtheorem{definition}[proposition]{Definition}
\newtheorem{remark}[proposition]{Remark}

\numberwithin{equation}{section}
\def\be{\begin{equation}}
\def\ee{\end{equation}}
\def\bp{\begin{proof}}
\def\ep{\end{proof}}

\def \no#1#2#3 {{\bf #1} (#3), #2.}
\def \eds#1#2#3 {#1, #2, #3.}

\title[Navier-Stokes-Brinkman-Forchheimer equation]
{The non-autonomous Navier-Stokes-Brinkman-Forchheimer equation with Dirichlet boundary conditions: dissipativity, regularity, and attractors}
\author[D. Stone  and  S. Zelik]
{Dominic Stone${}^1$ and Sergey Zelik${}^{1,2,3}$}

\address{${}^1$ University of Surrey, Department of Mathematics, Guildford, GU2 7XH, United Kingdom.}

\address{${}^2$ \phantom{e}School of Mathematics and Statistics, Lanzhou University, Lanzhou\\ 730000,
P.R. China}
\address{${}^3$ Keldysh Institute of Applied Mathematics, Moscow, Russia}

\email{d.stone@surrey.ac.uk}
\email{s.zelik@surrey.ac.uk}

\begin{document}

\begin{abstract}We give a comprehensive study of the 3D Navier-Stokes-Brinkman-Forchheimer equations in a bounded domain endowed with the Dirichlet boundary conditions and non-autonomous external forces. This study includes the questions related with the regularity of weak solutions, their dissipativity in  higher energy spaces and the existence of the corresponding uniform attractors
\end{abstract}

\subjclass[2010]{35B40, 35B42, 37D10, 37L25}
\keywords{Navier-Stokes equation, Brinkman-Forchheimer equation, regularity of solutions, non-autonomous attractors}
\thanks{ This work is partially supported by the grant 19-71-30004 of RSF (Russia)}
\maketitle
\tableofcontents

\section{Introduction}\label{s0}
We consider the  non-autonomous Navier-Stokes-Brink\-man-\-Forch\-heimer equations of the form
\begin{equation}\label{0.BF}
\begin{cases}
\Dt u+(u,\Nx)u+\Nx p+f(u)=\Dx u+g(t),\\
 \divv u=0,\ \  u\big|_{\partial\Omega}=0,\  \ u\big|_{t=0}=u_0
\end{cases}
\end{equation}
in a bounded smooth domain $\Omega\subset\R^3$ endowed with Dirichlet boundary conditions. Here $u=(u_1,u_2,u_3)$ and $p$ are an unknown velocity field and pressure respectively, $\Dx$ is the Laplacian with respect to the variable $x$, $g(t)$ are given external forces which satisfy
\begin{equation}\label{0.g}
g\in L^2_{loc}(\R,L^2(\Omega)),
\end{equation}
and $f(u)$ is a Forchheimer nonlinearity which satisfies the following assumptions:
\begin{equation}\label{0.f}
1.\ \  f\in C^1(\R^3,\R^3),\ \ \ 2. \ \  \kappa|u|^{r-1}-L\le f'(u)\le C(1+|u|^{r-1})
\end{equation}
for some positive constants $C$, $L$ and $\kappa$ and a growth exponent $r>1$.
\par
Equations of the form \eqref{0.BF} describe the fluid flows in porous media (see \cite{Au09,Br49,Mus37,Raj07,Stra08} and references therein) and can also appear under the study of tidal dynamics (usually without the inertial term, see \cite{Gor76,Ipa05,Lik81,Man08,Mar89,Moh18} and references therein).  The typical example of $f$ is
\begin{equation}\label{0.ff}
f(u)=\alpha u|u|^{r-1}+\beta u|u|-\gamma u,
\end{equation}
where $\alpha,\beta,\gamma\in\R$ and $\alpha>0$. Keeping in mind this example, in most part of our results, we will also require that the nonlinearity $f$ has a gradient structure
\begin{equation}\label{0.fgrad}
f(u)=\nabla_u F(u)
\end{equation}
for some $F\in C^2(\R^3,\R)$.
\par
Equations of the form \eqref{0.BF} under various assumptions on the nonlinearity and external forces are of a great current interest and many important and interesting results related with well-posedness of this problem, regularity and dissipativity of solutions, existence of attractors and determining functionals, etc. are obtained, see \cite{HR17,KZ,KZ21,MarT16,Wa08,You12} and references therein. In particular, the global existence of weak solutions of this problem can be proved exactly as in the case of usual Navier-Stokes problem (e.g. using the Galerkin approximations), see \cite{tem,KZ}. More interesting is that adding the extra Forchheimer term $f(u)$ provides the regularization of the problem and leads to the uniqueness of a weak solution if $r> 3$, see \cite{KZ,HR17} and also \cite{GalZ}. Recall that for the classical Navier-Stokes problem the global well-posedness of weak solutions is one of the Millennium problems (according to the Clay institute of mathematics). The case of equation \eqref{0.BF} with $r<3$ looks similar to the original Navier-Stokes equation and looks out of reach of the modern theory. The intermediate case $r=3$ may be easier to treat and some particular results in this direction are obtained, see e.g. \cite{HR17}, however, to the best of our knowledge, the global well-posedness of weak solutions for $r=3$ is also not known yet. For this reason, we restrict our analysis to the case $r>3$ only.
\par
The analysis of the further regularity of weak solutions of \eqref{0.BF} strongly depends on the type of boundary conditions. The situation is relatively simple in the whole space or in the case of periodic boundary conditions. Indeed, in this case we may multiply equation \eqref{0.BF} by $\Dx u$ and get an extremely important control of the higher energy norm of the solution $u$ (thanks to the conditions on $f'(u)$ and the cancellation of the pressure term), see \cite{Wa08,You12} and Remark \ref{Rem-reg} below.
\par
Things become much more complicated in  the case of Dirichlet boundary conditions where the term $(\Nx p,\Dx u)$ does not disappear and it is not clear how get the higher energy estimates. Indeed, the standard perturbation arguments based on the initial regularity of weak solutions allow us to treat the nonlinearity $f(u)$ only in the case $r\le\frac73$, but we need $r>3$ as mentioned above in order to treat properly the inertial term $(u,\Nx)u$. Thus, the standard approach fails to give more information for the further regularity of weak solutions no matter what the exponent $r$ is.
\par
This problem was partially resolved in \cite{KZ} for the case of autonomous right-hand side $g(t)\equiv g\in L^2(\Omega)$, using the so-called non-linear localization technique applied to the stationary Brinkman-Forchheimer problem:
\begin{equation}\label{0.sta}
f(u)+\Nx p-\Dx u=h,\ \ \divv u=0,\  u\big|_{\partial\Omega}=0,
\end{equation}
which gives us the control of the $H^2$-norm of a solution $u$ through the $L^2$-norm of $h$:
\begin{equation}\label{0.h2}
\|u\|_{H^2}\le Q(\|h\|_{L^2}).
\end{equation}
Then, the accurate analysis of the form of the function $Q$ allowed the authors to treat the inertial term as a perturbation and differentiation of the initial equation \eqref{0.BF} by $t$ followed by multiplication  by $\Dt u$ gave the control of the $L^2$-norm of $\Dt u(t)$ point-wise in time. This scheme allowed the authors to get finally the $H^2$-control of the norm of $u(t)$ point-wise in time and, since $H^2(\Omega)\subset C(\Omega)$, getting the further regularity of solutions becomes straightforward, see also \cite{KZ09,KZ21} for other non-trivial applications of the non-linear localization technique.
\par
Unfortunately, this method does not work for the case of non-autonomous external forces if we do not have enough regularity of $g(t)$ in time to differentiate it. Although, assuming in addition that
\begin{equation}\label{0.reg}
g\in H^1_{loc}(\R,L^2(\Omega))
\end{equation}
cures the problem, this assumption looks too restrictive, in particular, it excludes an important class of external forces which rapidly oscillate in time, e. g.
$$
g(t,x)=g_0(t/\eb)g_1(x),\ \ g_1\in L^2(\Omega),\ |\eb|\ll1.
$$
In this paper, we prefer not to use this extra assumption and try to understand instead what extra regularity of a weak solution $u(t)$ can be obtained if the external force $g$ satisfies \eqref{0.g} only.
We also believe that this understanding will be helpful for other problems related with equations of the form \eqref{0.BF}, in particular, for the global well-posedness of strong solutions in the intermediate case $r=3$.
\par
The central role in our study plays  assumption \eqref{0.fgrad} and the related second energy identity:
\begin{equation}\label{0.2en}
\frac{d}{dt}\(\frac12\|\Nx u\|^2_{L^2}+(F(u),1)\)+\|\Dt u\|^2_{L^2}=(g,\Dt u)-((u,\Nx)u,\Dt u).
\end{equation}
The only nontrivial term here, which prevents us to complete the estimate, is the one containing the inertial term. In order to handle it, we derive the Strichartz type estimate for the $L^2_{loc}(\R_+,L^\infty(\Omega))$-norm of the solution $u$. Obtaining this estimate is the most difficult and most technical part of the paper. We get this result applying the nonlinear localization technique directly to the non-stationary equation \eqref{0.BF} with the inertial term. In contrast to the previous results related with the stationary equation \eqref{0.sta}, we are unable to get the maximal regularity estimate for \eqref{0.BF} (i.e. to verify that every term in the equation belongs to the space $L^2_{loc}(\R_+,L^2(\Omega))$), but the obtained Strichartz type estimate mentioned above is enough for our purposes. This allows us to establish the following key result.

\begin{theorem}\label{Th0.main} Let the assumptions \eqref{0.f}, \eqref{0.fgrad} and \eqref{0.g} hold with $r>3$ and let $u_0\in\Phi$, where
\begin{equation}
\Phi:=H_0^1(\Omega)\cap L^{r+1}(\Omega)\cap\{\divv u_0=0\}.
\end{equation}
Then the unique weak solution $u(t)$ of problem \eqref{0.BF} satisfies the following estimate:
\begin{equation}\label{0.dis}
\|u(t)\|_{\Phi}^2\le Q\(\|u_0\|_{\Phi}^2e^{-\beta t}+\int_0^te^{-\beta(t-s)}\|g(s)\|^2_{L^2}\,ds\),
\end{equation}
for some positive constant $\beta$ and a monotone function $Q$ which are independent of $u_0$ and $t$.
\end{theorem}
In particular, if the right-hand side $g(t)$ is bounded as $t\to\infty$ in the sense that
\begin{equation}
\|g\|_{L^2_b(\R,L^2(\Omega))}:=\sup_{t\in\R}\|g\|_{L^2(t,t+1;L^2(\Omega))}<\infty,
\end{equation}
estimate \eqref{0.dis} gives us a uniform dissipative estimate
\begin{equation}\label{0.dis1}
\|u(t)\|_\Phi\le Q(\|u_0\|_\Phi)e^{-\beta t}+Q(\|g\|_{L^2_b})
\end{equation}
for some new monotone function $Q$ and this, in turn, allows us to study the long-time behaviour of solutions of \eqref{0.BF} in terms of uniform attractors for the cocycle generated by equation \eqref{0.BF}, see \cite{BV92,CLR12,ChVi02,KlR11,Zel22} and references therein.
 \par
 Note that, in contrast to the standard situation, it looks difficult here to establish the existence of a compact uniformly absorbing set for the corresponding cocycle based only on the compactness of Sobolev embeddings, so we utilize the energy method (see  \cite{Ball04}), the central part of which is to establish the energy identity \eqref{0.2en} for any weak solution of \eqref{0.BF}. Surprisingly, this is also  not  straightforward since the regularity of a weak solution provided by Theorem \ref{Th0.main} is not enough to justify the multiplication of \eqref{0.BF} by $\Dt u$. We overcome this difficulty using the convexity argument and the fact that the energy in \eqref{0.2en} is a compact perturbation of a uniformly convex functional. As a result, we prove that, under the mild extra assumption that the external forces $g$ are weakly normal, see \cite{Lu06,LWZ05,Lu07,MZS09,Z15} and also Definition \ref{Def4.w-n} below, the cocycle generated by solution operators of equation \eqref{0.BF} possesses a uniform attractor in the strong topology of $\Phi$ and is generated by all complete bounded solutions, see \S\ref{s4} below.
 \par
 The paper is organized as follows. In \S\ref{s1} we consider the stationary case of equation \eqref{0.BF} and refine the results of \cite{KZ} based on the non-linear localization technique. The obtained  results are crucial for the study of the non-autonomous case.
 \par
 In \S\ref{s2}, which is the key section of our paper, we apply the non-linear localization technique to the non-autonomous equation \eqref{0.BF} in order to get the above mentioned Strichartz type estimate and to give the proof of Theorem \ref{Th0.main}. We also give a number of extra regularity results which will be essentially used later.
 \par
 \S\ref{s3} is devoted to the verification of the energy identity
 \eqref{0.2en}. In order to do this, we utilize the convexity arguments and extend the ideas of \cite{MieZ14} to our case.
 \par
 Finally, in \S\ref{s4}, we briefly recall the main ideas and concepts of the attractors theory for non-autonomous dynamical systems and verify the existence of a uniform attractor for the cocycle related with equation \eqref{0.BF} in the phase space $\Phi$.

\section{The stationary case}\label{s1} In this section, we consider the following stationary version of the Brink\-man-Forchheimer equation:
\begin{equation}\label{1.BF}
\Dx u-\Nx p-f(u)=g,\ \ \divv u=0,\ \ u\big|_{\partial\Omega}=0.
\end{equation}
Here $u=(u_1(x),u_2(x),u_3(x))$ and $p=p(x)$ are an unknown velocity vector field and pressure respectively, $\Omega\subset \R^3$ is a bounded domain with a sufficiently smooth boundary, $f\in C^1(\R^3,\R^3)$ is a given nonlinearity which satisfies the following dissipativity and growth restrictions:
\begin{equation}\label{1.f}
\kappa(|u|^{r-1}+1)\le f'(u)\le C(1+|u|^{r-1})
\end{equation}
for some given positive constants $\kappa$ and $C$ and the exponent $r\in[1,\infty)$, and $g\in L^2(\Omega)$ are given external forces. We also assume that the mean of pressure is equal to zero
$$
\<p\>:=\frac1{|\Omega|}\int_\Omega p(x)\,dx=0.
$$
We are mainly interested in the case $r\gg1$, where the nonlinearity is not subordinated to the linear part of the equation. Note that this problem under the similar assumptions has been already considered
in \cite{KZ}, however, the estimates obtained there are not enough for our purposes, so in the present section, we refine them as well as deduce a number of new ones.
\par
We start with the standard energy estimate.
\begin{proposition}\label{Prop1.en} Let the nonlinearity $f$ satisfy \eqref{1.f} and $g\in L^2(\Omega)$. Let also $u$ be a sufficiently smooth solution of equation \eqref{1.BF}. Then, the following estimate holds:
\begin{equation}\label{1.en}
\|u\|^2_{H^1}+\|u\|^{r+1}_{L^{r+1}}+\|u\|_{W^{2,q}}^q+\|p\|_{W^{1,q}}^q+\|f(u)\|_{L^q}^q\le C\|g\|_{L^q}^q,
\end{equation}
where $\frac1{r+1}+\frac1q=1$ and $C$ is a positive constant which is independent on $u$.
\end{proposition}
\begin{proof} Multiplying equation \eqref{1.BF} by $u$, integrating over $\Omega$ and using the H\"older inequality, we get the energy identity
$$
\|\Nx u\|^2_{L^2}+(f(u),u)=(g,u)\le C_\kappa\|g\|^q_{L^q}+\frac\kappa2\|u\|_{L^{r+1}}^{r+1},
$$
where $(U,V)$ stands for the standard inner product in $[L^2(\Omega)]^3$. Using our assumptions on $f$, we deduce from here that
\begin{equation}\label{1.enn}
\|\Nx u\|^2_{L^2}+\|u\|_{L^{r+1}}^{r+1}+\|f(u)\|_{L^q}^q\le C\|g\|_{L^q}^q
\end{equation}
for some positive constant $C$. To complete the proof of the proposition, we write out equation \eqref{1.BF} as a linear equation
\begin{equation}
\Dx u-\Nx p=\tilde g:=g+f(u)
\end{equation}
and apply the maximal $L^q$-regularity result for solutions of the linear Stokes equation, see e.g. \cite{Galdi}. Together with the already obtained estimate for the $L^q$-norm of $f(u)$, this give us the desired control of the $W^{2,q}$-norm of $u$ as well as the $L^q$-norm of pressure $p$.
\end{proof}
At the next step, we want to obtain higher energy estimates for the solutions of \eqref{1.BF}. In the case of periodic boundary conditions, we get such estimates by multiplying the equation by $\Dx u$. However, this does not work at least in the straightforward way for the Dirichlet BC since the term $(\Nx p,\Dx u)$ becomes out of control because of non-trivial boundary terms, which appear under the integration by parts. Following \cite{KZ}, we use the nonlinear localization technique in order to overcome the problem.

\begin{theorem}\label{Th1.main} Let $u$ be a sufficiently smooth solution of equation \eqref{1.BF}. Then, the following partial regularity estimate holds:
\begin{equation}\label{1.3q}
\|u\|^2_{W^{1,3q}}\le C(1+\|g\|^2_{L^2}),
\end{equation}
where the constant $C$ is independent of $u$ and $g$. Moreover, the following conditional result holds:
\begin{equation}\label{1.ff}
\|f(u)\|_{L^{3q}}^q\le C\(1+\|g\|^q_{L^q}+(1+\|g\|^2_{L^2})^{2/3}(f'(u)\Nx u,\Nx u)^{1/3}\).
\end{equation}
\end{theorem}
\begin{proof} We divide the proof of this theorem on several steps.
\par
{\it Step 1. Interior regularity.} Let $\varphi(x)$ be a sufficiently smooth cut-off function which vanishes near the boundary and equal to one identically outside of the $\eb$-neighbourhood of $\partial\Omega$ and which satisfies the inequality
$$
|\Nx\varphi(x)|\le C_\eb\varphi(x)^{1/2},\ \ x\in\R^3.
$$
Since the domain $\Omega$ is smooth, such a function exists for every $\eb>0$.
\par
Let us multiply equation \eqref{1.BF} by $\sum \partial_{x_i}(\varphi\partial_{x_i}u)$ and transform the most complicated term containing pressure as follows:
\begin{multline}\label{1.4}
\(\Nx p,\sum \partial_{x_i}(\varphi\partial_{x_i}u)\)=-\(p, \sum_i\partial_{x_i}(\sum_j\partial_{x_j}\phi\partial_{x_i}u_j)\)=\\
=\sum_i\(\partial_{x_i}p,\sum_j\partial_{x_j}\varphi\partial_{x_i}u_j\)=\sum_i\(\partial_{x_i}p,\partial_{x_i}(\sum_j \partial_{x_j}\varphi u_j)\)-\\-\sum_{i,j}\(\partial_{x_i}p,\partial_{x_i,x_j}\varphi u_j\)=
-\(\Dx p,\Nx \varphi\cdot u\)- \sum_{i,j}\(\partial_{x_i}p,\partial_{x_i,x_j}\varphi u_j\).
\end{multline}
The last term in the RHS is easy to estimate using \eqref{1.en}:
$$
|\sum_{i,j}\(\partial_{x_i}p,\partial_{x_i,x_j}\varphi u_j\)|\le \|\Nx p\|_{L^q}\|u\|_{L^{r+1}}\le C\|g\|_{L^q}\|g\|_{L^q}^{1/r}=C\|g\|_{L^q}^q.
$$
To estimate the first term, we take the divergence from both sides of equation \eqref{1.BF} and obtain the Poisson equation for pressure:
\begin{equation}\label{1.5}
\Dx p=-\divv f(u)-\divv g.
\end{equation}
Thus, integrating by parts again, we arrive at
\begin{multline}
|(\Dx p,\Nx\varphi\cdot u)|=|(\divv f(u)+\divv g, \Nx\varphi \cdot u)|\le\\\le C(|g|,|\Nx u|)+C(|f(u)|,|u|)+(|f(u)|,|\Nx\varphi|\,|\Nx u|)\le\\\le C\|g\|^2_{L^2}+C(|u|^r+1,|\Nx u|\, |\Nx \varphi|).
\end{multline}
Using that $|\Nx \varphi|\le C\varphi^{1/2}$, we transform the last term as follows:
\begin{multline}
(|u|^r,|\Nx u|\,|\Nx\varphi|)\le C\(|u|^{(r-1)/2}\varphi^{1/2}|\Nx u|,|u|^{(r+1)/2}\)\le \\\le \beta(|u|^{r-1},\varphi|\Nx u|^2)+C_\beta\|u\|_{L^{r+1}}^{r+1}\le \beta(|u|^{r-1},\varphi|\Nx u|^2)+C_\beta(1+\|g\|^2_{L^2}),
\end{multline}
where $\beta>0$ is arbitrary.
Combining  the obtained estimates, we get
\begin{multline}\label{1.6}
|\(\Nx p,\sum\partial_{x_i}(\varphi\partial_{x_i}u)\)|\le \beta(\varphi|u|^{r-1},|\Nx u|^2)+C_\beta(1+\|g\|^2_{L^2}).
\end{multline}
It is now not difficult to complete the interior estimate. Indeed, multiplying equation \eqref{1.BF} by $\sum_i\partial_{x_i}(\varphi\partial_{x_i}u)$ and arguing in a standard way, with the help of \eqref{1.6}, we get
\begin{multline}
(\varphi,|\Dx u|^2)+(\varphi f'(u)\Nx u,\Nx u)\le C\|g\|^2+C\|\Nx u\|^2_{L^2}+\\+|\(\Nx p,\sum \partial_{x_i}(\varphi\partial_{x_i}u)\)|\le C_\beta(1+\|g\|^2_{L^2})+\beta(\varphi|u|^{r-1},|\Nx u|^2).
\end{multline}
Finally, due to our assumptions, $f'(u)\ge\kappa(|u|^{r-1}+1)$, so fixing $\beta<\kappa/2$, we arrive at the desired estimate
\begin{equation}\label{1.7}
(\varphi,|\Dx u|^2)+(\varphi |u|^{r-1},|\Nx u|^2)\le C(1+\|g\|^2_{L^2}),
\end{equation}
which gives the interior $H^2$-regularity.
\par
{\it Step 2. Boundary regularity: tangential directions.} We now look at a small neighbourhood of the boundary $\partial\Omega$. We introduce a family of orthogonal vector fields $\tau^1$, $\tau^2$ and $n$ in $\bar\Omega$ which are non-degenerate near the boundary and such that $\tau^i\big|_{\partial\Omega}$ are tangent vector fields and $n\big|_{\partial\Omega}$ is a normal vector field. Note that $n(x)$ is globally defined, but $\tau^i$ are in general only locally defined, so being pedantic we need to use the cut-off procedure in order to localize them. In order to avoid technicalities, we however will assume that they are globally defined. Let $\tau=(\tau_1,\tau_2,\tau_3)$
be one of the vectors $\tau^1$ or $\tau^2$. Then, we may define the corresponding differential operators
$$
\partial_\tau u:=\sum_{i=1}^3\tau_i\partial_{x_i}u,\ \ \partial^*_\tau u=-\sum_{i=1}^3\partial_{x_i}(\tau_i u).
$$
Our idea is to multiply equation \eqref{1.BF} by $\partial_\tau^*\partial_\tau u$ and get the $H^1$-norm of $\partial_\tau u$ analogously to Step 1 using the fact that
$\partial_\tau u\big|_{\partial\Omega}=0$. As before, we start with the most complicated term related with the pressure:
$$
(\Nx p, \partial_\tau^*(\partial_\tau u))=(\partial_\tau\Nx p,\partial_\tau u)=(\Nx(\partial_\tau p),\partial_\tau u)+([\partial_\tau,\Nx]p,\partial_\tau u).
$$
Here and below, we denote by $[D_1,D_2]$ the commutator of two differential operators $D_1$ and $D_2$, i.e. $[D_1,D_2]u:=D_1(D_2u)-D_2(D_1)u$.
Let us first estimate the term with the commutator using integration by parts and the fact that $u$ is divergence free:
\begin{multline}
([\partial_\tau,\Nx]p,\partial_\tau u)=\sum_j(\sum_i\tau_i\partial_{x_i}\partial_{x_j}p-
\partial_{x_j}(\tau_i\partial_{x_i}p),\sum_i\tau_i\partial_{x_i}u_j)=\\=
\sum_j(\sum_i\partial_{x_j}\tau_i\partial_{x_i}p,\sum_i\tau_i\partial_{x_i}u_j)=
-\sum_j\(\sum_k\partial_{x_k}(\tau_k\sum_i\partial_{x_j}\tau_i\partial_{x_i}p),u_j\)=\\=
-\(\sum_{i,j,k}\partial_{x_k}\tau_k\partial_{x_j}\tau_i\partial_{x_i}p,u_j\)-\(\sum_{i,j,k}\tau_k\partial_{x_j,x_k}\tau_i\partial_{x_i}p,u_j\)
-\\-\(\sum_{i,j,k}\tau_k\partial_{x_j}\tau_i\partial_{x_i}(\tau_k\partial_{x_k}p),u_j\)
+\(\sum_{i,j,k}\partial_{x_i}\tau_k\partial_{x_j}\tau_i\partial_{x_i}p,u_j\)=\\=
(A(x)\Nx(\partial_\tau p),u)+(B(x)\Nx p, u)
\end{multline}
for some smooth (and therefore bounded matrices $A$ and $B$). The remaining term is analogous, but even simpler
$$
(\Nx(\partial_\tau p),\partial_\tau u)=-(\partial_\tau p,[\partial_\tau,\divv]u)=(A_1(x)\Nx(\partial_\tau p),u)+(B_1(x)\Nx p,u)
$$
Thus, we have proved the following key relation:
\begin{equation}\label{1.8}
(\Nx p,\partial_\tau^*\partial_\tau u)=(\widetilde A(x)\Nx (\partial_\tau p),u)+(\widetilde B(x)\Nx p,u)
\end{equation}
for some smooth matrices $\widetilde A(x)$ and $\widetilde B(x)$. We now turn to the term containing the Laplacian:
\begin{multline}
(\Dx u,\partial_\tau^*\partial_\tau u)=(\partial_\tau\Dx u,\partial_\tau u)=(\Dx(\partial_\tau u),\partial_\tau u)+\\+([\partial_\tau ,\Dx]u,\partial_\tau u)=-\|\Nx(\partial_\tau u)\|^2_{L^2}+([\partial_\tau ,\Dx]u,\partial_\tau u).
\end{multline}
It only remains to estimate the term containing commutator. It is a sum of second and first order differential operators, so integrating by parts again in the second order part, we arrive at
$$
|([\partial_\tau ,\Dx]u,\partial_\tau u)|\le \beta\|\Nx(\partial_\tau u)\|^2_{L^2}+C_\beta\|\Nx u\|^2_{L^2}\le \beta\|\Nx(\partial_\tau u)\|^2_{L^2}+C_\beta\|g\|^2_{L^2}
$$
and, therefore, we end up with the estimate
\begin{equation}\label{1.9}
-(\Dx u,\partial_\tau^*\partial_\tau u)\ge\frac12\|\Nx(\partial_\tau u)\|^2_{L^2}-C\|g\|^2_{L^2}.
\end{equation}
The non-linear term is estimated in a standard way:
$$
(f(u),\partial_\tau^*\partial_\tau u)=(\partial_\tau f(u),\partial_\tau u)=(f'(u)\partial_\tau u,\partial_\tau u)\ge \kappa(|u|^{r-1},|\partial_\tau u|^2)
$$
and combining the obtained estimates together, we arrive at
\begin{equation}\label{1.10}
\|\Nx(\partial_\tau u)\|^2_{L^2}+(|u|^{r-1},|\partial_\tau u|^2)\le C(1+\|g\|^2_{L^2})+C|(\widetilde A(x)\Nx(\partial_\tau p),u)|.
\end{equation}
In order to estimate the term with pressure in the right-hand side, we differentiate equation \eqref{1.BF} along $\tau$ and write the obtained expression in the following form:
\begin{multline}\label{1.11}
\Dx(\partial_\tau u)-\Nx(\partial_\tau p)=\{\partial_\tau f(u)\}+\{\partial_\tau g\}+\\+\{[\Dx,\partial_\tau] u+\{[\Nx,\partial_\tau]p\}= : h_1+h_2+h_3,\ \\ \divv (\partial_\tau u)=[\divv,\partial_\tau] u:=H,\ \ \partial_\tau u\big|_{\partial\Omega}=0.
\end{multline}
From estimate \eqref{1.en}, we know that
$$
\|h_3\|_{L^q}\le C\|g\|_{L^q},
$$
where we have used that $[\partial_\tau,\Dx]$ and $[\partial_\tau,\divv]$ are second and first order operators respectively. Also $\|H\|_{W^{1,q}}\le C\|g\|_{L^q}$. We now decompose
$$
\partial_\tau u=v_1+v_2+v_3,\ \ \partial_\tau p=p_1+p_2+p_3,
$$
where
\begin{equation}\label{1.12}
\Dx(v_i)-\Nx p_i=h_i,\ \ \divv v_i=H_i,\ \ v_i\big|_{\partial\Omega}=0,
\end{equation}
where $H_3=H$ and $H_1=H_2=0$.
Then, due to the $L^q$ and $H^{-1}$ regularity estimates for the Stokes operator, see e.g. \cite{Galdi}, we have
$$
\|p_2\|_{L^2}\le C\|g\|_{L^2},\ \ \|\Nx p_3\|_{L^q}\le C\|g\|_{L^q},\ \ \|\Nx p_1\|_{L^q}\le C\|\partial_\tau f(u)\|_{L^q}.
$$
Inserting these estimates to \eqref{1.10}, we arrive at
\begin{multline}\label{1.13}
\|\Nx(\partial_\tau u)\|^2_{L^2}+(|u|^{r-1},|\partial_\tau u|^2)\le C(1+\|g\|^2_{L^2})+\\+C\|u\|_{L^{r+1}}\|\partial_\tau f(u)\|_{L^q}\le C(1+\|g\|^2_{L^2})+C\|\partial_\tau f(u)\|_{L^q}^q.
\end{multline}
It only remains to estimate the last term via the H\"older inequality:
\begin{multline}
\|\partial_\tau f(u)\|_{L^q}^q\le (|u|^{q(r-1)},|\partial_\tau u|^q)=
\(\(|u|^{(r-1)/2}|\partial_\tau u|\)^{q},|u|^{q(r-1)/2}\)\\\le
 (|u|^{r-1},|\partial_\tau u|^2)^{\frac q2}\(|u|^{r+1},1\)^{1-\frac q2}\le
 \beta(|u|^{r-1},|\partial_\tau u|^2)+C_\beta\|u\|_{L^{r+1}}^{r+1},
\end{multline}
where we have used that $1-\frac q2=1-\frac{r+1}{2r}=\frac{r-1}{2r}$ and $q(r-1)\frac{2r}{r-1}=2(r+1)$. Inserting the obtained estimate to the right-hand side of \eqref{1.13}, we end up with the desired estimate for tangential derivatives:
\begin{equation}\label{1.14}
\|\Nx(\partial_\tau u)\|^2_{L^2}+(|u|^{r-1},|\partial_\tau u|^2)\le C(1+\|g\|^2_{L^2}).
\end{equation}
{\it Step 3. Interpolation and regularity in normal directions.} Let $U:=\Nx u$. Then, estimates and \eqref{1.en} and \eqref{1.14} give us
\begin{equation}\label{1.U}
\|U\|_{W^{1,q}}+\|\partial_\tau U\|_{L^2}\le C(1+\|g\|_{L^2}).
\end{equation}
Recall that we are now doing estimates in a small neighbourhood of the boundary only (inside of the domain $\Omega$ we already have better control of the $H^2$-norm of $u$ from the proved interior estimate). Thus, using the partition of unity, we may fix  a small neighbourhood of containing a piece of the boundary and the straighten it by an appropriate diffeomorphism in such a way that in new coordinates $y=(y_1,y_2,y_3)$. The direction $y_1$ corresponds to the normal direction $\vec n$ and the variables $y_2$, and $y_3$  correspond to the tangent directions. Then, estimate \eqref{1.U} reads
\begin{equation}
\|\partial_{y_1}U\|_{L^q([0,1]^3)}+\|\partial_{y_2}U\|_{L^2([0,1]^3)}+\|\partial_{y_3}U\|_{L^2([0,1]^3)}\le C(1+\|g\|_{L^2}).
\end{equation}
After that we may use the interpolation for anisotropic Sobolev spaces:
\begin{multline}
W^{1,q}((0,1),L^q(0,1)^2)\cap L^2((0,1),W^{1,2}(0,1)^2)\subset\\\subset W^{\theta,s}((0,1),W^{1-\theta,s}(0,1)^2)\subset L^{3q}((0,1),L^{3q}(0,1)^2),
\end{multline}
where
 $$
 \frac1s=\frac\theta q+\frac{1-\theta}2,\ \frac1{3q}=\frac1s-\frac\theta1=\frac1s-\frac{1-\theta}2,
 $$
 so $\theta=\frac13$, see e.g. \cite{Rak} or \cite{BIN}.  Thus, reminding that $U=\Nx u$,we end up with the desired partial regularity estimate
 \begin{equation}\label{1.15}
 \|\Nx u\|_{L^{3q}}\le C\|u\|_{W^{2,q}}^{1/3}\|\Nx(\partial_\tau u)\|^{2/3}_{L^2}\le  C(1+\|g\|_{L^2}).
 \end{equation}
 Indeed, inside of the domain $\Omega$, we have better estimate with the exponent $6$ instead of $3q$ due to the proved interior regularity and estimate \eqref{1.15} in a small neighbourhood of the boundary follows from the interpolation mentioned above. Thus, estimate \eqref{1.3q} is proved and we only need to check \eqref{1.ff}. To this end, we first note that, due to the interior estimate \eqref{1.7}, our assumptions on $f$ and the Sobolev embedding $H^1\subset L^6$, we have
 \begin{multline}\label{1.fff}
 \|f(u)\|_{L^{3q}(\Omega_\eb)}^q\le C(1+\|u\|^{r+1}_{L^{3(r+1)}(\Omega_\eb)})\le\\\le C_1(1+\|u\|_{L^{r+1}(\Omega_\eb)}^{r+1}+\|\Nx (u^{(r+1)/2})\|^2_{L^2(\Omega)})\le\\\le C_2(1+\|g\|^q_{L^q(\Omega)}+(\varphi f'(u)\Nx u,\Nx u))\le C_3(1+\|g\|^2_{L^2(\Omega)}).
 \end{multline}
 Thus, we only need to look at a small neighbourhood near the boundary. Using the refined Sobolev  estimate
 $$
 \|v\|_{L^6([0,1]^3)}^3\le C\!\(\|v\|_{L^2([0,1]^3)}^3\!+\!\|\partial_{y_1}v\|_{L^2([0,1]^3)}\|\partial_{y_2}v\|_{L^2([0,1]^3)}\|\partial_{y_3}v\|_{L^2([0,1]^3)}\!\)
 $$
 which holds for every $v\in H^1_0((0,1)^3)$ and taking $v=|u|^{(r+1)/2}$, analogously to \eqref{1.fff}, we get
 \begin{multline}
 \|f(u)\|^q_{L^{3q}(\Omega)}\le C\(1+\|g\|^q_{L^q(\Omega)}+\right.\\ +\left.
 (f'(u)\partial_n u,\partial_n u)^{1/3}(f'(u)\partial_{\tau_1} u,\partial_{\tau_1} u)^{1/3}(f'(u)\partial_{\tau_2} u,\partial_{\tau_2} u)^{1/3}\).
\end{multline}
 Combining this estimate with \eqref{1.14}, we end up with the desired estimate \eqref{1.ff} and finish the proof of the theorem.
\end{proof}
We are now ready to complete the maximal $L^2$-regularity estimate for the solutions of equation \eqref{1.BF}.
\begin{corollary}\label{Cor1.h2} Let $u$ be a sufficiently smooth solution of equation \eqref{1.BF}. Then the following estimate holds:
\begin{equation}\label{1.fest}
\|u\|_{H^2}^2+\|\Nx p\|^2_{L^2}+\|f(u)\|^2_{L^2}\le C(1+\|g\|_{L^q})^{q}(1+\|g\|_{L^2})^s ,
\end{equation}
where the constant $C$ is independent of $u$ and $g$ and $s=\max\{2,\frac{4(2-q)}{3q-2}\}$.
\end{corollary}
\begin{proof} Using estimates \eqref{1.en}, \eqref{1.ff} and the interpolation inequality, we infer that
\begin{multline}\label{1.fl2}
\|f(u)\|^2_{L^2}\le\|f(u)\|_{L^q}^{\frac{3q-2}2}\|f(u)\|_{L^{3q}}^{\frac{3(2-q)}2}\le \\\le C(1+\|g\|_{L^q})^{\frac{3q-2}2}\(1+\|g\|^q_{L^q}+(1+\|g\|_{L^2}^2)^{\frac23}(f'(u)\Nx u,\Nx u)^{\frac13}\)^{\frac{3(2-q)}{2q}}.
\end{multline}
Since $1<q<2$, using the Young inequality, we may rewrite the last inequality in the form
\begin{multline}
\|f(u)\|^2_{L^2}\le C_1(1+\|g\|_{L^q})^2 +\\+(1+\|g\|_{L^q})^{\frac{3q-2}2}
(1+\|g\|_{L^2})^{\frac{2(2-q)}q}(f'(u)\Nx u,\Nx u)^{\frac{2-q}{2q}}\le\\\le C(1+\|g\|^2_{L^2})+
C_\nu(1+\|g\|_{L^q})^q(1+\|g\|_{L^2})^{\frac{4(2-q)}{3q-2}}+\nu(f'(u)\Nx u,\Nx u),
\end{multline}
where $\nu>0$ is arbitrary. Treating now the term $f(u)$ in \eqref{1.BF} as external forces and applying the maximal $L^2$-regularity estimate to the linear Stokes equation, we arrive at
\begin{multline}\label{1.endref}
\|u\|^2_{H^2}+\|\Nx p\|^2_{L^2}\le C\|g\|^2_{L^2}+\\+ C_\nu(1+\|g\|_{L^q})^q(1+\|g\|_{L^2})^{\frac{4(2-q)}{3q-2}}+\nu(f'(u)\Nx u,\Nx u).
\end{multline}
Moreover, from the equation \eqref{1.BF} we have
\begin{multline}\label{1.endf}
\|f(u)\|^2_{L^2}\le \|u\|^2_{H^2}+\|\Nx p\|^2_{L^2}+\|g\|^2_{L^2}\le C\|g\|^2_{L^2}+\\+  C_\nu(1+\|g\|_{L^q})^q(1+\|g\|_{L^2})^{\frac{4(2-q)}{3q-2}}+\nu(f'(u)\Nx u,\Nx u).
\end{multline}
Finally, in order to estimate the last term in the right-hand side, we multiply equation \eqref{1.BF} by $\Dx u$ and integrate over $x\in\Omega$. This gives
\begin{multline}
\|\Dx u\|^2_{L^2}+(f'(u)\Nx u,\Nx u)\le(\|g\|_{L^2}+\|\Nx p\|_{L^2})\|\Dx u\|_{L^2}\le\\\le \|g\|^2_{L^2}+\|\Dx u\|^2_{L^2}+\|\Nx p\|^2_{L^2}\le \|\Dx u\|^2_{L^2}+\\+C\|g\|^2_{L^2}+C_\nu(1+\|g\|_{L^q})^q(1+\|g\|_{L^2})^{\frac{4(2-q)}{3q-2}}+\nu(f'(u)\Nx u,\Nx u).
\end{multline}
Fixing $\nu>0$ small enough, we deduce that
$$
(f'(u)\Nx u,\Nx u)\le C\|g\|^2_{L^2}+C(1+\|g\|_{L^q})^q(1+\|g\|_{L^2})^{\frac{4(2-q)}{3q-2}},
$$
which together with \eqref{1.endref} and \eqref{1.endf} finish the proof of the corollary.
\end{proof}
We now briefly discuss the existence and uniqueness of a solution for \eqref{1.BF}. We start with weak solutions. To this end, we introduce the standard notation for the Navier-Stokes equations theory, namely, space $\Cal D_\sigma(\Omega)$ of divergence free $C^\infty$-smooth vector fields in $\Omega$ vanishing near the boundary as well as the spaces $\Cal H$ and $\Cal V$ which are the closure of $\Cal D_\sigma(\Omega)$ in $[L^2(\Omega)]^3$ and $[H^1(\Omega)]^3$ respectively. The space $\Cal V^{-1}$ is defined as a dual space to $\Cal V$ with respect to the duality generated by the standard inner product in $\Cal H$. Then, for every external force $g\in\Cal V^{-1}+L^q(\Omega)$, we say that $u$ is a weak solution of \eqref{1.BF} if $u\in\Cal V\cap L^{r+1}(\Omega)$ and satisfies equation \eqref{1.BF} in the sense of distributions, i.e.
\begin{equation}
(\Nx u,\Nx\varphi)+(f(u),\varphi)+(g,\varphi)=0
\end{equation}
for all test functions $\varphi\in\Cal D_\sigma(\Omega)$. As usual, see e.g. \cite{Galdi}, the pressure $p$ can be restored using $L^q$ and $H^{-1}$ regularity estimates for the linear Stokes operator, so we may also claim that $p\in L^2(\Omega)+W^{1,q}(\Omega)$ for any weak solution $u$.
\par
The existence of a weak solution is straightforward and can be done using, for instance, the Galerkin approximation method, see e,g, \cite{tem,Galdi}. The uniqueness also straightforward due to our monotonicity assumption on $f$ and the density of $\Cal D_\sigma(\Omega)$ in $\Cal V\cap L^{r+1}$ proved in \cite{Galdi} (see also \cite{GalZ} for an alternative proof). So, we discuss only the further smoothness of weak solutions.
\begin{corollary}\label{Cor1.sm} Let $g\in L^2(\Omega)$. Then, the unique weak solution $u$ of problem \eqref{1.BF} belongs to $H^2(\Omega)$, the corresponding pressure $p\in H^1(\Omega)$ and all of the estimates stated above hold.
\end{corollary}
\begin{proof} We first note that, due to the uniqueness of weak solutions, it is enough to construct a regular solution $u\in H^2(\Omega)$ for any $g\in L^2(\Omega)$ and this will automatically exclude the existence of weak non-smooth solutions. Since the nonlinear localization seems to not work on the level  of Galerkin approximations (at least we do not know how to obtain uniform estimates for higher norms), we need to use the alternative methods. For instance, we may use the continuation by the parameter methods based on the Leray-Schauder degree theory. Indeed, let us consider the family of equations of the form \eqref{1.BF} depending on a parameter $\eb\in[0,1]$:
\begin{equation}\label{1.BF1}
\Dx u-\Nx p-\eb f(u)=g,\ \ u\big|_{\Omega}=0.
\end{equation}
Using the Leray-Helmholtz projector $\Pi$ and the invertibility of the Stokes operator $A:=\Pi\Delta$, we may rewrite equation \eqref{1.BF1} in the equivalent form
\begin{equation}\label{1.BF2}
u=\eb A^{-1}\Pi f(u)+A^{-1}\Pi g.
\end{equation}
Let us assume that $g\in H^1$ and consider equation \eqref{1.BF2} as an equality in $H^3(\Omega)$. Then, due to the embedding $H^3(\Omega)\subset C^1(\Omega)$ and the $H^1\to H^3$ regularity of the solutions of the linear Stokes operator, see \cite{Galdi}, the right-hand side of \eqref{1.BF2} is a compact and continuous operator in $H^3(\Omega)$, so the Leray-Schauder degree theory works and we only need to obtain the uniform with respect to $\eb\in[0,1]$ estimates for the $H^3$-solutions of \eqref{1.BF2} or, which is the same, $H^3$-solutions of \eqref{1.BF1}.
\par
To get this estimate, we note that now we a priori know that the solution $u=u_\eb$ is in $H^3$ and the corresponding pressure $p_\eb\in H^2$ (by the regularity of the Leray-Helmholtz projector). Since all the above estimates related with nonlinear localizations obviously hold for such solutions, we may apply Corollary \ref{Cor1.h2} to equation \eqref{1.BF1}, which gives us the estimate of the $H^2$-norm of $u$. Moreover, it is not difficult to see that this estimate is unform with respect to $\eb$, so we end up with
$$
\|u_\eb\|_{H^2}\le Q(\|g\|_{L^2}),
$$
where the function $Q$ is independent of $\eb$. After that, using the embedding $H^2(\Omega)\subset C(\Omega)$, we get the control of the $H^1$-norm of $f(u)$ and, due to the regularity of the Stokes operator, we finally get the desired uniform a priori estimate in the $H^3$-norm. Thus, we have proved the existence of $H^3$-solution of problem \eqref{1.BF} under the extra condition that $g\in H^1$. Since the $H^2$-norm estimate of $u$ depends only on the $L^2$-norm of $u$, we may approximate a given external force $g\in L^2$ by a sequence $g_n\in H^1$, construct the associated solution $u_n$ and pass to the limit $n\to\infty$. This will give us the desired $H^2$-solution and finish the proof of the corollary.
\end{proof}
Let us now consider the stationary Navier-Stokes-Brinkman-Forchheimer equation:
\begin{equation}\label{1.BFS}
\Dx u-\Nx p-f(u)=g+(u,\Nx)u,\ \ \divv u=0,\ \ u\big|_{\partial\Omega}=0.
\end{equation}
Then, the weak solution of this problem is defined exactly as in the case of equation \eqref{1.BF}. Moreover, we have the energy estimate \eqref{1.enn} for such solutions due to the cancellation $((u,\Nx u),u)=0$. The existence of a weak solution can be obtained after that, e.g. by the Galerkin method. The next corollary tells us when every such solution is automatically smooth.
\begin{corollary} Let $g\in L^2(\Omega)$. Then any weak solution $u$ of problem \eqref{1.BFS} belongs to $H^2(\Omega)$ and the pair $(u,p)$ possesses the following estimate:
\begin{equation}\label{1.h2good}
\|u\|_{H^2}+\|p\|_{H^1}\le Q(\|g\|_{L^2}),
\end{equation}
where the function $Q$ is independent of $u$ and $g$.
\end{corollary}
\begin{proof} We first derive estimate \eqref{1.h2good} assuming that $u$ is smooth enough, e.g. $u\in H^2$. To this end, we interpret equation \eqref{1.BFS} and \eqref{1.BF} with the right-hand side $\tilde g:=g+(u,\Nx)u$ and apply estimate \eqref{1.fest} to it. In addition, analyzing the proof of this estimate, we see that the term containing $\|\tilde g\|_{L^q}$ in the right-hand side of it comes from the corresponding estimate of $\|f(u)\|_{L^q}$ and this control  we already proved for \eqref{1.BFS}, so we may replace $\tilde g$ by $g$ in this part of the estimate and arrive at
\begin{multline}
\|u\|^2_{H^2}+\|\Nx p\|_{H^1}^2\le\\\le C(1+\|g\|_{L^q})^q(1+\|\tilde g\|_{L^2})^s\le
Q(\|g\|_{L^2})(1+\|\,|u|\cdot|\Nx u|\,\|_{L^2})^s.
\end{multline}
Using the appropriate interpolation inequality together with \eqref{1.enn}, we estimate
$$
\|\,|u|\cdot|\Nx u|\,\|_{L^2}\le \|u\|_{L^\infty}\|u\|_{H^1}\le C\|u\|_{H^1}^{3/2}\|u\|_{H^2}^{1/2}\le C(1+\|g\|_{L^2})^{3q/4})\|u\|_{H^2}^{1/2}
$$
and, therefore,
$$
\|u\|^2_{H^2}+\|\Nx p\|_{H^1}^2\le Q(\|g\|_{L^2})(1+\|u\|^{s/2}_{H^2}).
$$
Since $s=\max\{2,\frac{4(2-q)}{3q-2}\}<4$ for $1\le q\le2$, the last estimate implies \eqref{1.h2good}. Thus, the $H^2$ a priori bound for a sufficiently smooth solution of \eqref{1.BFS} is obtained. The existence of such a solution can be verified, e. g. using again the Leray-Schauder degree theory, see the proof of Corollary \ref{Cor1.sm}, and we only need to check that any weak solution is actually smooth.
\par
In contrast to Corollary \ref{Cor1.sm}, we do not expect the uniqueness of a weak solution for problem \eqref{1.BFS} due to the presence of the non-monotone inertial term. In order to overcome this problem, we add an artificial term $Lu$ with big $L$ in both sides of equation \eqref{1.BFS} and consider the auxiliary problem
\begin{equation}\label{1.BFS1}
\Dx v-\Nx p-f(v)-Lv-(v,\Nx v)=g-Lu:=\bar g,
\end{equation}
where $u$ is a fixed weak solution of \eqref{1.BFS}. Then, since $\bar g\in L^2$ and the new nonlinearity $f(v)+Lv$ satisfies all of the assumptions posed on $f$, this problem possesses a smooth solution $v\in H^2$. On the other hand, $u$ is still a weak solution of this problem and we only need to check that $u=v$. This will be true  if $L$ is large enough to compensate the non-monotonicity of the inertial term. Indeed, let $w=u-v$. Then, this function solves
$$
\Dx w-\Nx \bar p-[f(w+v)-f(v)]-[(u,\Nx)w-(w,\Nx)v]-Lw=0.
$$
 Multiplying this equation by $w$, integrating by $x\in\Omega$ and using the monotonicity of $f$ together with the cancellation $((u,\Nx)w,w)=0$, we arrive at
 \begin{multline}
 \|\Nx w\|^2+L\|w\|^2\le |(w,\Nx)v,w)|\le \|\Nx v\|_{L^2}\|w\|_{L^4}^2\le\\\le C_1\|v\|_{H^1}\|w\|^{1/2}_{L^2}\|w\|_{H^1}^{3/2}\le C_2\|v\|^4_{H^1}\|w\|^2_{L^2}+\|\Nx w\|^2_{L^2}
 \end{multline}
 and we see that the uniqueness holds if $L>C_2\|v\|^4_{H^1}$. This finishes the proof of the corollary.
\end{proof}
\section{The non-stationary case}\label{s2}
In this section, we turn to study the non-stationary Brinkmann-Forch\-hei\-mer-Navier-Stokes equation of the form:
\begin{multline}\label{2.BF}
\Dt u+(u,\Nx)u+\Nx p+f(u)=\Dx u+g,\\ \divv u=0,\ u\big|_{t=0}=u_0,\  u\big|_{\Omega}=0,
\end{multline}
where $u_0\in \Cal V$ and
\begin{equation}\label{2.g}
g\in L^2_{loc}([0,\infty),L^2(\Omega)).
\end{equation}
Our general plan to tackle this problem is similar to what we did in Section \ref{s1}. Namely, we first obtain the basic energy estimate (by multiplying the equation by $u$) and after that improve the regularity of a solution using the nonlinear localization technique. In the case where $g$ is regular enough in time, e.g. $g\in C^1_b(\R,L^2(\Omega))$, one can get the control of the $C_b(\R_+,L^2(\Omega))$-norm of $\Dt u$ by differentiating the equation in time and multiplying it by $\Dt u$, see \cite{KZ} for the details. This will reduce the problem to the autonomous one which is considered in the previous section.
\par
However, the regularity of \eqref{2.g} is not sufficient to proceed in such a way and we need to apply the nonlinear localization technique to the {\it non-stationary} equation \eqref{2.BF} in direct way. As we will see below, the interior estimates as well as regularity in tangential directions can be extended to the non-stationary case with some extra technicalities related with the inertial term. In contrast to this, the obtained regularity in time is insufficient to obtain reasonable analogue of estimate \eqref{1.fest} and get the regularity in normal direction. For this reason, we failed to get full $H^2$-maximal regularity in the non-stationary case and do not know whether or not it holds (even in a "simple" case of periodic boundary conditions), but the obtained results are enough to establish the well posedness of \eqref{2.BF} in the phase space $\Cal V$ under the extra assumption that the nonlinearity $f$ is {\it gradient}, i.e.
\begin{equation}\label{2.fgr}
f(u)=\nabla_u F(u),
\end{equation}
for some $F\in C^2(\R^3)$. Again, we do not know whether or not the problem \eqref{2.BF} is globally well-posed without this assumption although some of our estimates remain true without it. We also relax slightly assumption \eqref{1.f} in order to include non-monotone nonlinearities. Namely, we assume from now on that
\begin{equation}\label{2.f}
\kappa |u|^{r-1}-L\le f'(u)\le C(1+|u|^{r-1})
\end{equation}
for some positive constants $\kappa$, $C$ and $L$.
\par
We start with the non-stationary analogues of basic energy estimates.
\begin{proposition}\label{Prop2.en} Let the assumptions \eqref{2.f} and \eqref{2.g} hold and let  $u$ be a sufficiently smooth solution of \eqref{2.BF}. Then, the following estimate holds:
\begin{multline}\label{2.en}
\|u(t)\|^2_{\Cal H}+\|u(t)\|^q_{W^{2(1-1/q),q}}+\!\!\int_0^te^{-\alpha(t-s)}\(\|\Nx u(s)\|^2_{L^2}+\|u(s)\|^q_{W^{2,q}}+\right.\\\left.+\|\Dt u(s)\|^q_{L^q}+\|u(s)\|^{r+1}_{L^{r+1}}+\|f(u(s))\|_{L^q}^q+\|\Nx p(s)\|^q_{L^q}\)\,ds\le\\\le C\(\|u_0\|^2_{\Cal H}+\|u_0\|^q_{W^{2(1-1/q),q}}\)e^{-\alpha t}+C\(1+\int_0^te^{-\alpha(t-s)}\|g(s)\|_{L^q}^q\,ds\),
\end{multline}
where the positive constants $C=C_\alpha$ and $0<\alpha<\alpha_0$ can be chosen arbit\-rarily (where $\alpha_0$   is small enough). All constants  are independent of $t$ and~$u$.
\end{proposition}
\begin{proof} We multiply equation \eqref{2.BF} by $u$ and integrate over $x\in\Omega$ to get
\begin{equation}\label{2.eneq}
\frac12\frac d{dt}\|u(t)\|^2_{L^2}+\|\Nx u(t)\|^2_{L^2}+(f(u(t)),u(t))=(g(t),u(t)).
\end{equation}
Then, using assumption \eqref{2.f} together with the H\"older and Poincare inequality, we arrive at
\begin{multline}
\frac d{dt}\|u(t)\|^2_{L^2}+\alpha\|u(t)\|^2_{L^2}+\\+\alpha(\|u(t)\|^2_{H^1}+
\|u(t)\|_{L^{r+1}}^{r+1}+\|f(u(t))\|_{L^q}^q)\le C(1+\|g\|_{L^q}^q),
\end{multline}
where $C$ and $\alpha$ are positive constants which are independent of $t$ and $u$. Integrating this inequality, we have
\begin{multline}\label{2.en-bas}
\|u(t)\|^2_{L^2}+\alpha\int_0^te^{-\alpha(t-s)}\(\|u(s)\|^2_{H^1}+
\|u(s)\|_{L^{r+1}}^{r+1}+\|f(u(s))\|_{L^q}^q\)\,ds\\\le \|u(0)\|^2_{L^2}e^{-\alpha t}+ C\(\alpha^{-1}+\int_0^te^{-\alpha(t-s)}\|g(s)\|_{L^q}^q\,ds\).
\end{multline}
Thus, the basic energy estimate is proved (note that the assumption $r\ge3$ is nowhere used here). To complete estimate \eqref{2.en}, we need to rewrite equation \eqref{2.BF} in the form of a linear non-stationary Stokes equation:
\begin{equation}\label{2.Sto}
\Dt u-\Dx u+\Nx p=g(t)-f(u(t))-(u(t),\Nx)u(t):=g_u(t)
\end{equation}
and to apply the standard $L^q$-maximal regularity estimate to this equation. This gives
\begin{multline}\label{2.maxSt}
\|u(t)\|^q_{W^{2(1-1/q),q}}+\\+\int_0^te^{-\alpha(t-s)}\(\|\Dt u(s)\|^q_{L^q}+\|u(s)\|^q_{W^{2,q}}+\|\Nx p(s)\|^q_{W^{1,q}}\)\,ds\le\\\le C\|u(0)\|^q_{W^{2(1-1/q),2}}e^{-\alpha t}+C\int_0^te^{-\alpha(t-s)}\|g_u(s)\|^q_{L^q}\,ds,
\end{multline}
where $C$ and $\alpha$ are some constants and $|\alpha|$ is small enough, see e.g. \cite{Solo}. So, it only remains to estimate the norm of $g_u(s)$ in the right-hand side. Moreover, the term containing $f(u)$ is already estimated and we only need to estimate the inertial term. To this end, we need the assumption $r\ge3$ which allows us to use the H\"older inequality in the form
\begin{multline}\label{2.iner}
\|(u,\Nx)u\|_{L^q}^q\le \|u\|_{L^{\frac{2q}{2-q}}}^q\|\Nx u\|_{L^2}^q\le\\\le\|\Nx u\|^2_{L^2}+\|u\|_{L^{\frac{2q}{2-q}}}^{\frac{2q}{2-q}}\le C(1+\|\Nx u\|^2_{L^2}+\|u\|_{L^{r+1}}^{r+1}),
\end{multline}
since $\frac{2q}{2-q}=\frac{2(r+1)}{r-1}\le r+1$ if $r\ge3$. Thus, the inertial term is also controlled by \eqref{2.en-bas} and this estimate together with \eqref{2.maxSt} give the desired estimate \eqref{2.en} and finish the proof of the proposition.
\end{proof}
Analogously to the stationary case, we define a weak solution $u(t)$ of problem \eqref{2.BF} as a function which belongs to the space
$$
C([0,\infty),\Cal H_w)\cap L^2_{loc}([0,\infty),\Cal V)\cap L^{r+1}_{loc}([0,\infty).L^{r+1}),
$$
 which satisfies \eqref{2.BF} in the sense of distributions, i.e.
 $$
 -\<u,\Dt\varphi\> +\<\Nx u,\Nx\varphi\>+\<f(u),\varphi\>+\<(u,\Nx)u,\varphi\>=\<g,\varphi\>
 $$
 for all $\varphi\in C_0^\infty(\R_+\times\Omega)$ with $\divv \varphi(t) = 0$. Here and below $\<u,v\>:=\int_{\R}(u(t),v(t))\,dt$ and $C([0,\infty),\Cal H_w)$ means the space of $\Cal H$-valued functions $u(t)$, $t\in\R_+$, which are continuous in time in the weak topology of $\Cal H$. We summarize the known facts about the existence and uniqueness of such solutions in the following proposition.
\begin{proposition}\label{Prop2.weak}
Let the nonlinearity $f$ satisfy assumption \eqref{2.f} for some $r\ge 1$. Then, for every $u_0\in\Cal H$ and every $g\in L^q_{loc}([0,\infty),L^q(\Omega))$, problem \eqref{2.BF} possesses at least one weak solution $u$ which satisfies the energy estimate \eqref{2.en-bas}. If, in addition, $r>3$, then the weak solution of this problem is unique. Moreover, if $r>3$ and $u_0\in \Cal H\cap W^{2(1-1/q),q}(\Omega)$, then estimate \eqref{2.en} holds for the weak solution of \eqref{2.BF}.
\end{proposition}
\begin{proof} Indeed, the existence of a solution follows in a standard way from estimate \eqref{2.en-bas} using, e.g. Galerkin approximations. The uniqueness of a solution is known for $r>3$ only. First of all, we use this assumption in order to check that the inertial term is in $L^q$, see \eqref{2.iner}. After that, we justify the multiplication of equation \eqref{2.BF} by $u$ as well as the multiplication of for difference of two weak solutions $u$ and $v$ by $u-v$ and this gives us the following identity for the difference $w:=u-v$:
\begin{equation}\label{2.endif}
\frac12\frac d{dt}\|w\|^2_{L^2}+\|\Nx w\|^2_{L^2}+([f(w+v)-f(v)],w)+ ((w,\Nx)v,w)=0,
\end{equation}
see \cite{KZ,GalZ} for the details. Then, using the integration by parts, we estimate the inertial term as follows
$$
|((w,\Nx)v,w)|\le (|w|\cdot |v|,|\Nx w|)\le \|\Nx w\|^2+(|w|^2,|v|^2).
$$
Moreover, using assumption \eqref{2.f} on the nonlinearity, we arrive at
$$
(f(u)-f(v),w)\ge \kappa'(|u|^{r-1}+|v|^{r-1},v^2)\ge (|v|^2,|w|^2)-C_{r,\kappa}\|w\|^2_{L^2}
$$
for some positive constants $\kappa'$ and $C_{r,\kappa}$ (here we have used that $r>3$ again). Inserting these estimates into \eqref{2.endif}, we have
$$
\frac12\frac d{dt}\|w\|^2_{L^2}\le C_{r,\kappa}\|w\|^2_{L^2}
$$
and the Gronwall inequality gives us that $w=0$ which proves the uniqueness.  Finally, if $u_0\in \Cal H\cap W^{2(1-1/q),q}(\Omega)$, we may apply the maximal $L^q$-regularity estimate \eqref{2.maxSt} to the linear Stokes equation \eqref{2.Sto} and verify that the unique weak solution \eqref{2.Sto} satisfies \eqref{2.en}. This finishes the proof of the proposition.
\end{proof}
We are now ready to state and prove the main result of this section, which gives the global well-posedness of problem \eqref{2.BF} in the phase space $\Cal V\cap L^{r+1}(\Omega)$.
\begin{theorem}\label{Th2.main} Let $u_0\in\Phi:=\Cal V\cap L^{r+1}(\Omega)$, the nonlinearity $f$ satisfy \eqref{2.f} and \eqref{2.fgr} for some $r>3$ and let $g$ satisfy \eqref{2.g}. Then the weak solution $u(t)\in\Phi$ for all $t\ge0$ and  satisfies the following estimate:
\begin{equation}\label{2.main}
\|u(t)\|^2_\Phi\!+\!\!\int_{t-1}^t\!\|\Dt u(s)\|^2_{L^2}ds\!\le\! Q\(\!\|u_0\|_\Phi^2e^{-\beta t}\!\!+\!\!\int_0^te^{-\beta(t-s)}\|g(s)\|^2_{L^2}ds\!\),
\end{equation}
where $\|u\|_\Phi^2:=\|u\|_{\Cal V}^2+\|u\|_{L^{r+1}}^{r+1}$, $Q$ is a monotone function and $\beta$ is a positive constant, both are independent of $t$ and $g$ and we put $0$ instead of $t-1$ if $t\le1$. Moreover, if we only assume that $u_0\in\Cal H$, then $u(t)\in\Phi$ for all $t>0$ and the following estimate holds for $t\ge1$:
\begin{equation}\label{2.main1}
\|u(t)\|^2_\Phi\!+\!\!\int_{t-1}^t\!\|\Dt u(s)\|^2_{L^2}ds\!\le\! Q\(\!\|u_0\|_{\Cal H}^2e^{-\beta t}\!\!+\!\!\int_0^te^{-\beta(t-s)}\|g(s)\|^2_{L^2}ds\!\)\!.
\end{equation}
\end{theorem}
\begin{proof} Of course, the analogue of the smoothing property \eqref{2.main1} holds for all $t>0$ with the function $Q$ depending also on $1/t$ and we state the smoothing property for $t\ge1$ just in order to avoid extra technicalities.
\par
 We first assume that $u$ is a sufficiently smooth solution of \eqref{2.BF}, for instance,
\begin{equation}\label{2.reg}
u\in C_{loc}([0,\infty),H^2(\Omega))\cap L^2_{loc}([0,\infty),H^3(\Omega)),
\end{equation}
and derive the desired estimates for it. We divide the proof on several steps.
\par
{\it Step 1. Interior estimates.} This step is almost identical to Step1 in the proof of Theorem \ref{Th1.main}. Indeed, let $\varphi$ be the same as in that step. Multiplying equation \eqref{2.BF} by
$-\sum_{i}\partial_{x_i}(\varphi\partial_{x_i}u)$, we arrive at
\begin{multline}\label{2.int}
\frac d{dt}(\varphi,|\Nx u|^2)+(\varphi,|\Dx u|^2)+\kappa(|u|^{r-1}\varphi,|\Nx u|^2)\le \\\le C(1+\|g(t)\|^2_{L^2}+\|\Nx u\|^2_{L^2}+\|u\|^{r+1}_{L^{r+1}}+\|\Nx p\|^q_{L^q})+\\+2(\Dx p,\Nx\varphi\cdot u)-2((u,\Nx)u,\divv(\varphi\Nx u)).
\end{multline}
Thus, we just have an extra term in the right-hand side of \eqref{2.int} which is related with pressure and which should be properly estimated and also we now have
$$
\Dx p=-\divv f(u)-\divv g-\divv(u,\Nx)u
$$
with the extra term related with the divergence of the inertial term (in comparison with \eqref{1.5}). Both of the extra terms are not difficult to control. Indeed,
\begin{multline}\label{2.extra}
|((u,\Nx)u,\divv(\varphi\Nx u))|\le C(|u|\cdot |\Nx u|,\varphi|\Dx u|)+\\+C(|u|,\varphi^{1/2}|\Nx u|^2)|\le \frac14(\varphi,|\Dx u|^2)+\frac14\kappa(|u^{r-1},|\Nx u|^2)+C\|\Nx u\|^2_{L^2}.
\end{multline}
The extra term in the expression for the Laplacian of pressure $p$, after inserting it into the right-hand side of \eqref{2.int} gives us the term
$$
|(\divv(u,\Nx)u,\Nx\varphi\cdot u)|=|((u,\Nx) u, \Nx (\Nx\varphi\cdot u))|,
$$
which can be estimated exactly as in \eqref{2.extra}. Combining all of the estimates together, we arrive at
\begin{multline}\label{2.int1}
\frac d{dt}(\varphi,|\Nx u|^2)+\frac12(\varphi,|\Dx u|^2)+\frac\kappa2(|u|^{r-1}\varphi,|\Nx u|^2)\le \\\le C(1+\|g(t)\|^2_{L^2}+\|\Nx u\|^2_{L^2}+\|u\|^{r+1}_{L^{r+1}}+\|\Nx p\|^q_{L^q}).
\end{multline}
Finally, integrating this relation in time and using \eqref{2.en}, we arrive at the desired interior estimate
\begin{multline}\label{2.int2}
(\varphi,|\Nx u(t)|^2)+\\+\alpha\int_0^t e^{-\alpha(t-s)}\((\varphi,|\Dx u(s)|^2)+(|u|^{r-1}\varphi,|\Nx u|^2)\)\,ds\le \\\le C\|u_0\|^2_{\Cal V}e^{-\alpha t}+C\(1+\int_0^te^{-\alpha(t-s)}\|g(s)\|^2_{L^2}\,ds\).
\end{multline}
for some positive constants $C$ and $\alpha$.
\par
{\it Step 2. Boundary regularity: tangential direction.} Analogously to Section \ref{s1}, we multiply equation \eqref{2.BF} by $\partial^*_\tau\partial_\tau u$ and integrate over $x\in\Omega$. Then, in comparison with \eqref{1.10}, we will have an extra good term $\frac12\frac d{dt}\|\partial_\tau u\|^2_{L^2}$ as well as the term $((u,\Nx)u,\partial_\tau^*\partial_\tau u)$ related with the extra inertial term, which can be estimated as follows:
\begin{multline}
|((u,\Nx)u,\partial_\tau^*\partial_\tau u)|\le|((\partial_\tau u,\Nx) u,\partial_\tau u)|+|((u,\Nx)\partial_\tau u, \partial_\tau u)|+\\+C(|u|\cdot|\partial_\tau u|,|\Nx u|)\le |((\partial_\tau u,\Nx) u,\partial_\tau u)|+C(|u|^2,|\partial_\tau u|^2)+C\|\Nx u\|^2_{L^2}.
\end{multline}
Furthermore, integrating by parts, we get
\begin{multline}
|((\partial_\tau u,\Nx) u,\partial_\tau u)|\le C|(|\partial_\tau u|\cdot |u|,|\Nx(\partial_\tau u)|)|\le\\\le \eb\|\Nx(\partial_\tau u)\|^2+C_\eb(|u|^2,|\partial_\tau u|^2).
\end{multline}
Thus, using again that $r>3$, we arrive at
$$
|((u,\Nx)u,\partial_\tau^*\partial_\tau u)|\le \eb\(\|\Nx(\partial_\tau u)\|^2_{L^2}+(|u|^{r-1},|\partial_\tau u|^2)\)+C_\eb\|\Nx u\|^2_{L^2},
$$
where $\eb>0$ is arbitrary,
and, therefore, this extra term is under the control and, arguing as in the stationary case, we arrive at
\begin{multline}\label{2.26}
\frac d{dt}\|\partial_\tau u\|^2_{L^2}+\alpha\(\|\Nx(\partial_\tau u)\|^2_{L^2}+(|u|^{r-1},|\partial_\tau u|^2)+\|\partial_\tau f(u)\|^q_{L^q}\)\le\\\le C(1+\|g(t)\|_{L^2}^2)+
C\|\Nx u(t)\|^2_{L^2}+C\|u(t)\|^{r+1}_{L^{r+1}}+C|(A(x)\Nx(\partial_\tau p),u)|.
\end{multline}
Integrating this inequality in time and using  \eqref{2.en-bas}, after the standard transformations, we arrive at
\begin{multline}\label{2.tan}
\sup_{s\in[0,t]}\{e^{-\beta(t-s)}\|\partial_\tau u(s)\|^2\}+\kappa\int_0^te^{-\beta(t-s)}\(\|\Nx(\partial_\tau u(s))\|^2_{L^2}+\right.\\\left.+(|u(s)|^{r-1},|\partial_\tau u(s)|^2)+\|\partial_\tau f(u(s))\|^q_{L^q}\)ds\le C\|u_0\|^2_{\Cal V}e^{-\beta t}+C +\\+C\int_0^te^{-\beta(t-s)}\|g(s)\|^2_{L^2}\,ds+C\int_0^te^{-\beta(t-s)}|(A(x)\Nx(\partial_\tau p(s)),u(s))|\,ds,
\end{multline}
where $\kappa$, $\beta$ and $C$ are some positive constants, which are independent of $t$ and $u$. Thus, we only need to estimate the term, containing pressure. To this end, we introduce a function $G=G(t)$ as a solution of the linear Stokes equation:
$$
\Dt G-\Dx G+\Nx p_G=g(t),\ \ G\big|_{t=0}=u_0,\ \ \divv G=0,\ \ G\big|_{\partial\Omega}=0.
$$
Then, using the $L^2$-maximal regularity estimate for the linear Stokes equation, see \cite{Solo}, we end up with
\begin{multline}\label{2.Gest}
\|G(t)\|^2_{L^2}+\\+\int_0^te^{-\beta_1(t-s)}\(\|\Dt G(s)\|^2_{L^2}+\|G(s)\|^2_{H^2}+\|\Nx p_G(s)\|^2_{L^2}\)\,ds\le\\\le C\|u_0\|_{\Cal V}^2e^{-\beta_1(t-s)}+C\int_0^te^{-\beta_1(t-s)}\|g(s)\|^2_{L^2}\,ds
\end{multline}
for some positive constants $\beta_1>\beta$, and $C$ (and we also have the $L^q$-version of this estimate). We also introduce a new function $v:=u-G$ which also solves the linear Stokes equation:
\begin{equation}\label{2,v}
\Dt v-\Dx v+\Nx p_v=-f(u)-(u,\Nx)u,\ \ v\big|_{t=0}=0,\ \ \divv v=0.
\end{equation}
Differentiating this equation with respect to $\tau$ and denoting $w:=\partial_\tau v$ and $p_w:=\partial_\tau p_v$, we arrive at
\begin{multline}\label{2.w}
\Dt w-\Dx w+\Nx p_w=\partial_\tau f(u)-\partial_\tau (u,\Nx)u-[\Dx,\partial_\tau]v+\\+[\Nx,\partial_\tau]p_v,\ \ w\big|_{\partial\Omega}=0,\ w\big|_{t=0}=0,\  \divv w=[\divv,\partial_\tau]v:=H(t).
\end{multline}
Our plan is to apply the $L^q$-maximal regularity estimate to this linear non-homogeneous Stokes equation. Indeed, from estimates \eqref{2.en} and \eqref{2.Gest}, we only know that
\begin{multline}\label{2.H-est}
\|H(t)\|_{W^{1,q}}+\|\Dt H(t)\|_{W^{-1,q}}\le C(\|u(t)\|_{W^{2,q}}+\|\Dt u(t)\|_{L^q}+\\+\|\Dt G(t)\|_{L^q}+\|G(t)\|_{W^{2,q}})
\end{multline}
and the $L^q$-norm in time from the left-hand side is under the control. However, this is not enough in general to get the maximal $L^q$-regularity estimate (in general, we need $\Dt H$ to belong to $L^q(0,t;L^q)$, see e. g. \cite{FShi} for the counterexample). Fortunately, the function $H(t)$ in \eqref{2.w} has a special structure which allows us to overcome this problem. Namely, it is not difficult to see that
$$
H(t)=\divv W(t)+h(t),\ \ W_i:=-v\cdot \Nx \tau_i,\ \ h:=v\cdot \Nx\divv\tau.
$$
Important is that $W\big|_{\partial\Omega}=0$. Therefore, we may subtract the function $W$ from the solution $w$ and get a new linear Stokes problem for the function $\bar w:=w-W$:
\begin{multline}\label{2.bw}
\Dt \bar w-\Dx \bar w+\Nx p_w=\partial_\tau f(u)-\partial_\tau (u,\Nx)u-[\Dx,\partial_\tau]v+\\+[\Nx,\partial_\tau]p_v-g_W:=g_{\bar w},\ \ w\big|_{\partial\Omega}=0,\ w\big|_{t=0}=0,\  \divv w=h(t),
\end{multline}
where $g_W:=\Dt W-\Dx W$. Since both $W$ and $h$ are proportional to $v$, we have the control of the $L^q$-norms of $G_W$ and $\Dt h$ from \eqref{2.en}. Moreover, since
$$
\|[\Dx,\partial_\tau]v\|_{L^q}\le C\|v\|_{W^{2,q}},\ \ \|[\Nx,\partial_\tau] p_v\|_{L^q}\le C\|p_v\|_{W^{1,q}}
$$
and $p=p_G+p_v$, all terms with commutators are also controlled by \eqref{2.en}. We actually need not to estimate $\partial_\tau f(u)$ since this term is presented in the left-hand side of \eqref{2.tan} and will be finally cancelled out. However, we still need to estimate the most complicated term related with the inertial term in \eqref{2.bw}, but we prefer to postpone this estimate and first complete the exclusion of pressure. To this end, we apply the $L^q$-regularity estimate to problem \eqref{2.bw}, see \cite{FShi} and get
\begin{multline}
\int_0^te^{-\beta_1(t-s)}\|\Nx p_w(s)\|_{L^q}^q\,ds\le\\\le  C\int_0^t e^{-\beta_1(t-s)}\(\|g_{\bar w}(s)\|^q_{L^q}+\|h(s)\|^q_{W^{1,q}}+\|\Dt h(s)\|^q_{L^q}\)\,ds\le\\\le
C\int_0^t e^{-\beta_1(t-s)}\(\|\partial_\tau f(u(s))\|^q_{L^q}+\|\partial_\tau(u(s),\Nx)u(s))\|^q_{L^q}\)\,ds+\\+
C\int_0^te^{-\beta_1(t-s)}\(\| v(s)\|^q_{W^{2,q}}+\|\Dt v(s)\|^q_{L^q}+\|\Nx p_v(s)\|_{L^q}^q\)\,ds.
\end{multline}
Using the obvious estimate
\begin{equation}\label{2.weights}
\int_0^te^{-\beta_1(t-s)}\(\int_0^se^{-\alpha(s-\tau)}|U(\tau)|\,d\tau\)\,ds\le C^*\!\!\int_0^te^{-\beta_1(t-s)}|U(s)|\,ds,
\end{equation}
where $\alpha>\beta_1>0$ and the constant $C^*$ depends only on $\alpha$ and $\beta$, together with estimate \eqref{2.en} and the $L^q$-version of estimate \eqref{2.Gest}, we arrive at
\begin{multline}\label{2.pres}
\int_0^te^{-\beta_1(t-s)}\|\Nx p_w(s)\|_{L^q}^q\,ds\le \\\le
C\int_0^t e^{-\beta_1(t-s)}\(\|\partial_\tau f(u(s))\|^q_{L^q}+\|\partial_\tau(u(s),\Nx)u(s))\|^q_{L^q}\)\,ds+\\+
C\|u_0\|_{\Cal V}e^{-\beta_1 t}+C\(1+\int_0^te^{-\beta_1(t-s)}\| g(s)\|^q_{L^q}\,ds\).
\end{multline}
We are now ready to return to estimation of the last term in \eqref{2.tan}. Namely,
 \begin{multline}
 |(A(x)\Nx(\partial_\tau p),u)|\le |(\Nx(\partial _\tau p_G),A^*(x)u)|+|(A(x)\Nx p_w,u)||\le\\\le C\(\|u\|^2_{H^1}+\|\Nx p_G\|_{L^2}^2\)+\nu\|\Nx p_w\|^q_{L^q}+C_\nu\|u\|^{r+1}_{L^{r+1}},
\end{multline}
where $\nu>0$ is arbitrary. Using the obtained estimates \eqref{2.pres} and \eqref{2.Gest} together with \eqref{2.weights}, we exclude the pressure from \eqref{2.tan} and get
\begin{multline}\label{2.tan1}
\sup_{s\in[0,t]}\big\{e^{-\beta(t-s)}\|\partial_\tau u(s)\|^2\big\}+\int_0^te^{-\beta(t-s)}\(\|\Nx(\partial_\tau u(s))\|^2_{L^2}+\right.\\\left.+(|u(s)|^{r-1},|\partial_\tau u(s)|^2)+\|\partial_\tau f(u(s))\|^q_{L^q}\)ds\le C_\nu\|u_0\|^2_{\Cal V}e^{-\beta t}+C_\nu +\\+C_\nu\int_0^te^{-\beta(t-s)}\|g(s)\|^2_{L^2}\,ds+
\nu \int_0^te^{-\beta(t-s)}\|\partial_\tau((u(s),\Nx)u(s))\|_{L^q}^q\,ds.
\end{multline}
Moreover, the last term in this estimate can be further simplified. Namely,
\begin{equation}\label{2.inn}
\partial_\tau(u,\Nx)u=(\partial_\tau u,\Nx)u+(u,\Nx)\partial_\tau u+(u,[\Nx,\partial_\tau])u
\end{equation}
and
$$
\|(u,\Nx)\partial_\tau u\|_{L^q}^q\le \|\Nx\partial_\tau u\|_{L^2}^q\|u\|_{L^{\frac{2q}{2-q}}}^q\le
 \nu\|\Nx\partial_\tau u\|^2_{L^2}+C_\nu\|u\|_{L^{\frac{2q}{2-q}}}^{\frac{2q}{2-q}}.
 $$
 Since $\frac{2q}{q-2}=\frac{2(r+1)}{r-1}<r+1$ if $r>3$, this term is under the control. The third term in the right-hand side of \eqref{2.inn} can be estimated analogously using the fact that the commutator $[\Nx,\partial_\tau]$ is a first order differential operator. So, we only need to estimate the first term. We will do this with the help of \eqref{1.15}, the H\"older inequality and the fact that $\frac{3q}2\le2$, namely,
 \begin{multline}\label{2.str-est}
 \|\,|\partial_\tau u|^q\cdot|\Nx u|^q\|_{L^1}\le \|\partial_\tau u\|^q_{L^{3q/2}}\|u\|_{W^{1,3q}}^q\le\\\le C\|\partial_\tau u\|_{L^2}^q\|\Nx(\partial_\tau u)\|_{L^2}^{2q/3}\|\Dx u\|_{W^{2,q}}^{q/3}\le \|\Nx(\partial_\tau u)\|^2_{L^2}+\\+C\|\partial_\tau u\|_{L^2}^{\frac{3q}{3-q}}\|u\|^{\frac q{3-q}}_{W^{2,q}}\le \|\Nx(\partial_\tau u)\|^2_{L^2}+C\|\partial_\tau u\|_{L^2}^{\frac{5q-4}{3-q}}\|\partial_\tau u\|_{L^2}^{\frac{2(2-q)}{3-q}}\|u\|_{W^{2,q}}^{\frac q{3-q}}.
 \end{multline}
 Crucial for us is the fact that $\frac{5q-4}{3-q}<2$ for $q<\frac{10}7$ and therefore, due to our assumptions, $q<\frac43<\frac{10}7$, so the number $m:=2- \frac{5q-4}{3-q}$ is always positive. In addition, the first term in the right-hand side of \eqref{2.str-est} is not dangerous since it is absorbed by the corresponding term in the left hand side of \eqref{2.tan1}, so we only need to estimate the integral
 $$
 I:=\int_0^te^{-\beta(t-s)}\|\partial_\tau u(s)\|_{L^2}^{\frac{5q-4}{3-q}}\|\Nx u(s)\|_{L^2}^{\frac{2(2-q)}{3-q}}\|u(s)\|_{W^{2,q}}^{\frac q{3-q}}\,ds.
 $$
 To this end, we use that the  exponents $\frac{2(2-q)}{3-q}$ and $\frac q{3-q}$ are such that, due to the Young inequality
 $$
 \|\Nx u(s)\|_{L^2}^{\frac{2(2-q)}{3-q}}\|u(s)\|_{W^{2,q}}^{\frac q{3-q}}\le \|\Nx u(s)\|^2_{L^2}+\|u(s)\|^{q}_{W^{2,q}}
 $$
 and, therefore,
 \begin{multline}
 I\le C\(\sup_{s\in[0,t]}\big\{e^{-\beta(t-s)}\|\partial_\tau u(s)\|^2_{L^2}\big\}\)^{1-m/2}\times\\\times
 \int_0^te^{-\beta(t-s)/2}\(\|\Nx u(s)\|^2_{L^2}+\|u(s)\|^q_{W^{2,q}}\)\,ds\le\\\le
 \nu\sup_{s\in[0,t]}\big\{e^{-\beta(t-s)}\|\partial_\tau u(s)\|^2_{L^2}\big\}+\\+C_\nu\(\int_0^te^{-\beta(t-s)/2}\(\|\Nx u(s)\|^2_{L^2}+\|u(s)\|^q_{W^{2,q}}\)\,ds\)^{\frac2m}.
\end{multline}
Inserting this estimate to the right-hand side of \eqref{2.tan1} and using estimate \eqref{2.en} with $\alpha=\beta/2$, we finally end up with the following estimate:
   \begin{multline}\label{2.tan2}
\sup_{s\in[0,t]}\big\{e^{-\beta(t-s)}\|\partial_\tau u(s)\|^2\big\}+\int_0^te^{-\beta(t-s)}\(\|\Nx(\partial_\tau u(s))\|^2_{L^2}+\right.\\\left.+(|u(s)|^{r-1},|\partial_\tau u(s)|^2)+\|\partial_\tau f(u(s))\|^q_{L^q}\)ds\le\\\le C\(\|u_0\|^2_{\Cal V}e^{-\beta t/2}+1 +\int_0^te^{-\beta(t-s)/2}\|g(s)\|^2_{L^2}\,ds\)^{\tilde m},
\end{multline}
where $\tilde m:=\max\{1,2/m\}$. This finishes the boundary regularity estimate in tangential directions.
\par
{\it Step 3. Key interpolation estimate.} Namely, we start with the Gagliardo-Nirenberg inequality:
$$
\|u\|_{L^\infty}\le C\|u\|_{L^{r+1}}^{1/4}\|u\|_{W^{1,3q}}^{3/4},
$$
since $\frac34-3\(\frac1{4(r+1)}+\frac 1{4q}\)=0$, see e.g. \cite{MaO} or \cite{Adams}. This inequality,
 together with  \eqref{1.15}, give us
\begin{multline}\label{2.w-est}
\|u\|_{L^\infty}^2\le C\|u\|_{L^{r+1}}^{1/2}\|u\|_{W^{2,q}}^{1/2}\|\Nx(\partial_\tau u)\|_{L^2}\le\\\le C(\|u\|_{L^{r+1}}^{r+1}+\|u\|^q_{W^{2,q}}+\|\Nx(\partial_\tau u)\|^2_{L^2}),
\end{multline}
where we have used that $1=\frac1{2(r+1)}+\frac1{2q}+\frac12$. Using estimates \eqref{2.en} and
\eqref{2.tan2}, we get the desired estimate
\begin{multline}\label{2.inf-est}
\int_0^te^{-\beta(t-s)}\|u(s)\|^2_{L^\infty}\,ds\le\\\le C\(\|u_0\|_{\Cal V}^2e^{-\beta t/2}+\int_9^te^{-\beta(t-s)/2}\|g(s)\|^2_{L^2}\,ds\)^{\tilde m}.
\end{multline}
{\it Step 4. $\Phi$-energy estimate.} Till that moment, we have nowhere used our extra assumption that the nonlinearity $f$ is gradient, but it is essentially used at this step. Indeed, we now multiply equation \eqref{2.BF} by $\Dt u$ and integrate over $x\in\Omega$. This gives
\begin{multline}\label{2.gron}
    \frac d{dt}\(\frac12\|\Nx u\|^2_{L^2}+(F(u),1) + L \| u \|_{L^2}^2 \)+\frac12\|\Dt u\|^2_{L^2}=-\frac12\|\Dt u\|^2_{L^2} \\ - ((u,\Nx u),\Dt u) +(g + 2Lu,\Dt u)\le   C(\|g\|^2_{L^2}+\|u\|_{L^2}^2)+\\ +C \|u\|_{L^\infty}^2\(\frac12\|\Nx u\|^2+(F(u),1) + L \|u \|_{L^2}^2\),
\end{multline}
where $L$ is such that  $F(u)+L|u|^2\ge0$ (it exists due to assumption \eqref{2.f}). Note that assumption \eqref{2.f} implies also that
\begin{equation}\label{2.F}
\kappa_2|u|^{r+1}-C_2\le F(u)\le \kappa_1|u|^{r+1}+C_1
\end{equation}
for some positive constants $\kappa_i$ and $C_i$. Therefore, estimate \eqref{2.inf-est} allows us to apply the Gronwall inequality to \eqref{2.gron} and to get the following control:
\begin{multline}\label{2.non-dis}
\|u(t)\|_{H^1}^2+\|u(t)\|^{r+1}_{L^{r+1}}\le Ce^{C\int_0^t\|u(s)\|_{L^\infty}^2\,ds}\times\\\times\(\|u(0)\|_{H^1}^2+\|u(0)\|^{r+1}_{L^{r+1}}+
\int_0^t\(\|u(s)\|^2_{L^2}+\|g(s)\|^2_{L^2}\)\,ds\).
\end{multline}
We will use this estimate for $0\le t\le1$ only since it is growing in time even if $g(t)$ is bounded and, therefore, is not convenient for study the attractors. Combining \eqref{2.non-dis} with \eqref{2.inf-est} and \eqref{2.en}, we arrive at the following estimate for $t\in[0,1]$:
\begin{equation}\label{2.nd}
\|u(t)\|_{H^1}^2\!+\!\|u(t)\|^{r+1}_{L^{r+1}}\le Q\!\(\|u(0)\|_{H^1}^2\!+\!\|u(0)\|^{r+1}_{L^{r+1}}+\!\!\int_0^t\|g(s)\|^2_{L^2}\,ds\)
\end{equation}
for some monotone increasing function $Q$. For $t\ge1$, we will use the following smoothing estimate:
\begin{equation}\label{2.smo-h1}
\|u(1)\|_{H^1}^2\!+\!\|u(1)\|^{r+1}_{L^{r+1}}\le Q\!\(\|u(0)\|_{L^2}^2+\!\!\int_0^t\|g(s)\|^2_{L^2}\,ds\)
\end{equation}
for some new monotone function $Q$. This estimate is a standard corollary of \eqref{2.nd} and \eqref{2.en}. Indeed, from \eqref{2.en}, we know that
$$
\int_0^1\(\|u(t)\|_{H^1}^2\!+\!\|u(t)\|^{r+1}_{L^{r+1}}\)\,dt\le C\(\|u_0\|_{L^2}^2+1+\int_0^1\|g(t)\|^2_{L^2}\,dt\).
$$
Therefore, there exists $t_0\in[0,1]$ (depending on $u$) such that
$$
\|u(t_0)\|_{H^1}^2\!+\!\|u(t_0)\|^{r+1}_{L^{r+1}}\,  \le C\(\|u_0\|_{L^2}^2+1+\int_0^1\|g(t)\|^2_{L^2}\,dt\).
$$
Applying after that estimate \eqref{2.nd} on the time interval $t\in[t_0,1]$, we arrive at \eqref{2.smo-h1}. In turn, combining estimates \eqref{2.nd} (on interval $t\in[0,1]$) with estimate \eqref{2.smo-h1} (which will give us the estimate of the $H^1\cap L^{r+1}$-norm of $u(t)$ through the $L^2$-norm of $u(t-1)$, $t\ge1$) together with the dissipative estimate \eqref{2.en}, we end up with the desired estimates \eqref{2.main} and \eqref{2.main1}. The estimate for the $L^2$-norm of $\Dt u$ in them follows immediately by integrating \eqref{2.gron} in time.
\par
{\it Step 5. $\Phi$-regularity of solutions.} We recall that all previous estimates were derived assuming that $u(t)$ is a sufficiently smooth solution of \eqref{2.en}, for instance, satisfying \eqref{2.reg},
will be enough to justify all of them (here we used that $H^2\subset C$ in 3D, so all terms related with the nonlinearity are under the control). To get such a regular solution, we approximate the external force $g$ and the initial data $u_0$ by the sequences $g_n$ and $u_0^n$ of smooth functions. Then, as proved in \cite{KZ}, there exist an $L^\infty_{loc}(\R_+,H^2)$ smooth solution $u_n(t)$ of problem \eqref{2.en} where the initial data $u_0$ and the external force $g$ are replaced by $u_0^n$ and $g_n$ respectively. Moreover, since $f\in C^1$, it is easy to see that $\Dt f(u_n)\in L^2_{loc}(\R_+,L^2)$ and, therefore, the standard regularity result for the linear Stokes equation gives us that \eqref{2.reg} are satisfied. For this reason, the solutions $u_n$ satisfy estimates \eqref{2.main} and \eqref{2.main1} uniformly with respect to $n$. Passing to the limit $n\to\infty$, we see that the limit unique solution $u(t)$ of problem \eqref{2.BF} also satisfies these estimates. Thus, Theorem \ref{Th2.main} is proved.
\end{proof}
The next corollary is gives us slightly stronger version of estimate \eqref{2.inf-est}. This improved estimate will be used later for the attractors theory.
\begin{corollary}\label{Cor2.reg} Let the assumptions of Theorem \ref{Th2.main} hold. Then the solution $u$ of problem \eqref{2.BF} satisfies the following estimate:
\begin{multline}\label{2.imp}
\int_{t-1}^t\|\Nx u(s)\|^2_{W^{1,3q}}+\|u(s)\|_{L^\infty}^{8/3}\,ds\le\\\le Q\(\|u_0\|_{\Phi}^2e^{-\beta t/2}+\int_0^te^{-\beta(t-s)/2}\|g(s)\|^2_{L^2}\,ds\)
\end{multline}
for some monotone function $Q$ and positive constant $\beta$.
\end{corollary}
\begin{proof} We rewrite equation \eqref{2.BF} in the form of a stationary problem \eqref{1.BF} at every fixed time $t$, namely,
\begin{multline}\label{2.BFs}
-\Dx u+\Nx p+f(u)+Lu=\tilde g(t):=\\=g(t)-\Dt u(t)-(u(t),\Nx)u(t)+L u(t),\ \ \divv u=0,\ u\big|_{\partial\Omega}=0.
\end{multline}
Moreover, due to the proved theorem, the $L^2_{loc}(\R_+,L^2)$-norm of $\tilde g$ is under the control. For this reason, we may apply estimates \eqref{1.en} and \eqref{1.3q} to this equation (recall that the modified non-linearity $\tilde f(u):=f(u)+Lu$ satisfies \eqref{1.f}. Then, integrating estimates
 \eqref{1.en} (in a power $2/q$) and \eqref{1.3q}, we arrive at
 \begin{multline}\label{2.inm}
 \int_{t-1}^t\|u(s)\|^2_{W^{2,q}}+\|u(s)\|^2_{W^{1,3q}}\,ds\le\\\le Q\(\|u_0\|_{\Phi}^2e^{-\beta t/2}+\int_0^te^{-\beta(t-s)/2}\|g(s)\|^2_{L^2}\,ds\).
\end{multline}
Thus, we have got the part of the desired estimate related with the $W^{1,3q}$-norm of $u$. In order to get the remaining part, we improve estimate \eqref{2.w-est} using that we now have estimate for the $L^2$-norm in time for $\|u(t)\|_{W^{2,q}}$ and also the $L^\infty$ norm in time for $\|u(t)\|_{L^{r+1}}$:
\begin{multline}\label{2.s-est}
\|u\|_{L^\infty}^{8/3}\le C\|u\|_{L^{r+1}}^{2/3}\|u\|_{W^{2,q}}^{2/3}\|\Nx(\partial_\tau u)\|_{L^2}^{4/3}\le\\\le C\|u\|_{L^{r+1}}^{2/3}(\|u\|^2_{W^{2,q}}+\|\Nx(\partial_\tau u)\|^2_{L^2}),
\end{multline}
This estimate, together with \eqref{2.main} and \eqref{2.inm} completes the desired estimate \eqref{2.imp} and finishes the proof of the corollary.
\end{proof}
\begin{remark}\label{Rem-reg} In the case of periodic boundary conditions, we are able to multiply equation \eqref{2.BF} by $\Dx u$ and the pressure term will still disappear. This immediately gives us the result of Theorem \ref{Th1.main} with linear function $Q$ and also the control of the integral of $(f'(u)\Nx u,\Nx u)$. Then, from \eqref{1.ff}, we get the control of the $L^q$-norm in time of $\|f(u)\|_{L^{3q}}$. Combining this with the interpolation inequality \eqref{1.fl2} and the $L^\infty$-control for $\|f(u(t))\|_{L^q}$, we arrive at the control of the $L^{\frac{4q}{3(2-q)}}_{loc}(\R_+,L^2)$-norm of $f(u)$. In particular, if $r\le 5$. then
$\frac{4q}{3(2-q)}\ge2$ and we get the $L^2$ space-time regularity of $f(u)$ which together with the $L^2$-regularity estimate for the linear Stokes problem gives us the maximal $L^2$-regularity for equation \eqref{2.BF}, namely,
\begin{equation}\label{2.max-l2}
u,\Dt u,\Nx p, \Dx u, f(u)\in L^2_{loc}(\R_+,L^2(\Omega).
\end{equation}
Unfortunately, we do not know how to get such a regularity for $r>5$ (even in the case of periodic boundary conditions). Instead, applying the anisotropic
$L^{\frac{4q}{3(2-q)}}(L^2)$-regularity estimate for the linear Stokes equation, we have a weakened version of the regularity \eqref{2.max-l2} for $r>5$:
\begin{equation}\label{2.per-l2}
\Dt u,\Dx u\in L^2_{loc}(\R_+,L^2(\Omega)),\ \ f(u),\Nx p\in L^{\frac{4q}{3(2-q)}}_{loc}(\R_+,L^2(\Omega)).
\end{equation}
We now return to the case of Dirichlet boundary conditions. In this case, we do not have the control over the term $(f'(u)\partial_n u,\partial_n u)$, so the regularity \eqref{2.per-l2} should be further weakened. In this case, we need to use estimate \eqref{1.fest} instead. Note that for $r\le 5$, we have the exponent $s=2$ there and, therefore, using again that the term with the $L^q$-norm of $g$ comes from the estimate of $\|f(u(t))\|_{L^q}$ and we have the $L^\infty$-norm control for this term, we see that, for $r\le5$, we have the maximal regularity \eqref{2.max-l2} for the case of Dirichlet boundary conditions as well. In the case $r>5$, we have $s=\frac{4(2-q)}{3q-2}$ and, therefore, the obtained regularity reads
\begin{equation}\label{2.dir-l2}
\Dt u\in L^2_{loc}(\R_+,L^2(\Omega)),\ \Dx u,\Nx p,f(u)\in L^{\frac{3q-2}{2-q}}_{loc}(\R_+,L^2(\Omega)).
\end{equation}
Since $\frac{4q}{3(2-q)}>\frac{3q-2}{2-q}$ for $r>5$, the regularity \eqref{2.dir-l2} available for Dirichlet boundary conditions and $r>5$ is indeed weaker than in the case of periodic boundary conditions. We do not know whether this is a drawback of the method or a real loss of regularity for the case of Dirichlet boundary conditions.
\par
We also note that the most difficult part in the proof of the key Theorem \eqref{Th2.main} was to estimate the $L^q$-norm of $\partial_\tau((u,\Nx)u)$. There is an alternative way to treat this term, namely, we may rewrite the inertial term in the form of $\divv(u\otimes u)$ and then, up to lower order terms, we will need to estimate $\divv(u\otimes\partial_\tau u+\partial_\tau u\otimes u)$. Since we have the term $(|u|^{r-1} \partial_\tau u,\partial_\tau u)$ in the right-hand side of \eqref{2.26}, the $L^2_{loc}(\R_+,H^{-1})$-norm of  $\partial_\tau((u,\Nx)u)$ is under the control. In order to complete the estimate in this way, we need the maximal regularity for the non-stationary Stokes equation in $H^{-1}$. Unfortunately, in contrast to the stationary case, this maximal regularity fails in general and we need some extra assumptions on the right-hand side in order to restore it. It would be interesting to check whether or not the inertial term satisfies these extra assumptions.
\par
We finally note that the function $Q$ in the key estimate \eqref{2.main} has an exponential growth rate
($Q(z)\sim e^{Cz^{\tilde m}}$), which somehow indicates that the nonlinearities in the problem \eqref{2.BF} are {\it critical} for all $r>3$. Since there is no such a criticality for the case of periodic boundary conditions, we expect that this also may be the drawback of the method.
\end{remark}

\section{The higher energy identity}\label{s3}
The aim of this section is to verify the  $H^1$-energy equality, which can be formally obtained by multiplying equation \eqref{2.BF} by $\Dt u$ and which is the key technical tool for verifying the asymptotic compactness of the dynamical processes associated with this equation. The problem here is that the proved regularity of a solution does not allow to interpret the terms $(\Dx u,\Dt u)$ and $(f(u),\Dt u)$ in the sense of distributions, only the inner product with their difference  $(\Dx u-f(u)-\Nx p,\Dt u)$ is well-defined. Therefore, we need some accuracy with verifying the identity
\begin{multline}\label{3.en}
\frac d{dt} E(u(t))=-(\Dt u(t)\!-\!(u(t),\Nx) u(t)+g,\Dt u(t)):=(\Dt u(t),H_u(t)),\\ E(u):=\frac12\|\Nx u\|^2_{L^2}+(F(u),1).
\end{multline}
 Following \cite{MieZ14}, we use the convexity arguments to verify \eqref{3.en}.
\begin{theorem}\label{Th3.main} Let the assumptions of Theorem \ref{Th2.main} hold and let $u(t)$ be the solution $u(t)$ of \eqref{2.BF}. Then the function $t\to E(u(t))$ is absolutely continuous and the identity \eqref{2.en} holds for almost all $t\in\R_+$ as well as in the sense of distributions.
\end{theorem}
\begin{proof} Note first of all that, without loss of generality, we may assume that $f'(u)\ge0$, so $F(u)$ is convex.  We use the following identities, which can be verified by straightforward computations:
\begin{multline}\label{3.h1}
-(\Dx u(t),u(t+h)-u(t))+\frac12\|\Nx u(t+h)-\Nx u(t)\|^2_{L^2}=\\=\frac12\(\|\Nx u(t+h)\|^2_{L^2}-\|\Nx u(t)\|^2_{L^2}\)=\\=-(\Dx u(t+h),u(t+h)- u(t))-\frac12\|\Nx u(t+h)-\Nx u(t)\|^2_{L^2}
\end{multline}
and
\begin{multline}\label{3.h2}
(f(u(t+h)),u(t+h)-u(t)))+\\+\int_0^1(f'(su(t+h)+(1-s)u(t))(u(t+h)-u(t),u(t+h)-u(t))\,ds=\\=
(F(u(t+h)),1)-(F(u(t)),1)=\\=(f(u(t)),u(t+h)-u(t)))-\\-\int_0^1(f'(su(t+h)+(1-s)u(t))(u(t+h)-u(t),u(t+h)-u(t))\,ds.
\end{multline}
Note that all terms in \eqref{3.h1} and \eqref{3.h2} make sense as functions from $L^1_{loc}(\R_+)$ and these identities and the regularity of a solution $u(t)$ is enough to justify them. Taking a sum of these two identities and using that $f'(u)\ge0$, we end up with the following two-sided inequality:
\begin{multline}\label{3.h}
(-\Dx u(t+h)+f(u(t)),u(t+h)-u(t)) \le\\\le E(u(t+h))-E(u(t))\le (-\Dx u(t)+f(u(t)),u(t+h)-u(t)).
\end{multline}
Inserting the expression for $-\Dx u+f(u)$ from \eqref{2.BF} to \eqref{3.h} and integrating over $t\in[S,T]$, we arrive at
\begin{multline}\label{3.hh}
\frac1h\int_S^{T}(H_u(t+h),u(t+h)-u(t))\,dt\le\\\le
\frac1h\int_T^{T+h}\!\!E(u(t))\,dt-\frac1h\int_S^{S+h}\!\!E(u(t))\,dt\le \frac1h\int_S^T\!(H_u(t),u(t+h)-u(t))\,dt.
\end{multline}
We know that $H_u\in L^2_{loc}(\R_+,L^2)$, $E(u)\in L^1_{loc}(\R_+)$ and $u\in W^{1,2}(\R_+,L^2)$, therefore, we may pass to the limit $h\to0$ in \eqref{3.hh} and get that, for almost all $S,T\in\R_+$, $S\le T$, we have the integral identity
\begin{equation}\label{3.a-en}
E(u(T))-E(u(S))=\int_S^T(H_u(t),\Dt u(t))\,dt.
\end{equation}
To finish the proof of the theorem, we need to remove the condition that \eqref{3.a-en} holds for {\it almost all} $S\le T$ only. To this end, let us assume that the energy {\it inequality}
\begin{equation}\label{3.e-in}
E(u(T))-E(u(S))\le\int_S^T(H_u(t),\Dt u(t))\,dt.
\end{equation}
is already verified for {\it all}  $0\le S<T$. We recall that $u(t)$ is weakly con\-ti\-nu\-ous as a function of time with values in $\Phi$, therefore the values of $E(u(t))$ are well-defined for all $t\in\R_+$. Moreover, since the function $u\to E(u)$ is convex, we have
\begin{equation}\label{3.con}
E(u(t))\le \liminf_{s\to t}E(u(s))
\end{equation}
for all $t\in\R_+$. In particular, this property, together with \eqref{3.e-in} imply that the function $t\to E(u(t))$  is continuous from the right for all $t\in\R_+$. In turn, this ensures us that the energy equality \eqref{3.a-en} holds for all $0\le S<T$. Indeed, since \eqref{3.a-en} holds almost everywhere, we may find sequences $S_n\to S$, $S_n\ge S$, and $T_n\to T$, $T_n\ge T$ such that \eqref{3.a-en} holds for
$S_n$ and $T_n$ for all $n$. Passing to the limit $n\to\infty$, we see that it holds for $S$ and $T$ as well. This finishes the proof of the theorem since the energy equality \eqref{3.a-en}, which holds for  all $S$ and $T$ is equivalent to the absolute continuity of the function $t\to E(u(t))$ and equality \eqref{3.en}.
\par
Thus, we only need to verify the energy inequality \eqref{3.e-in}. Moreover, due to the uniqueness of a solution for problem \eqref{2.BF}, it is sufficient to verify it for $S=0$ only (the general case will follow just by replacing the initial time $t=0$ by $t=S$).  To do this, we approximate the initial data $u_0\in\Phi$ and the external force $g\in L^2_{loc}(\R_+,L^2)$ by smooth functions $u_0^n$ and $g_n$ respectively. Let $u_n(t)$ be the corresponding solutions of \eqref{2.BF}. Then, analogously to the end of the proof of Theorem \eqref{Th2.main}, on the one hand, the solutions $u_n(t)\in C_{loc}(\R_+,H^2)$ and therefore, the energy equality \eqref{3.a-en} holds for them for every $T\ge0$, i.e.
\begin{equation}\label{3.a-enn}
E(u_n(T))-E(u_n(0))=\int_0^T(H_{u_n}(t),\Dt {u_n}(t))\,dt.
\end{equation}
On the other hand, the solutions $u_n$ satisfy the key estimates obtained in section \ref{s2} uniformly with respect to $n$. These estimates guarantee that, in particular, $u_n(t)\to u(t)$ weakly in $\Phi$ for all $t\ge0$ and, by the choice of the initial data, we have the strong convergence at $t=0$. Therefore, without loss of generality we have
$$
E(u(0))=\lim_{n\to\infty}E(u_n(0)),\ \ E(u(t))\le\lim_{n\to\infty}E(u_n(t))
$$
for all $t\in\R$. In order to pass to the limit in the term containing $H_{u_n}$, we note that we also know that $\Dt u_n\to \Dt u$ weakly in $L^2_{loc}(\R_+,L^2)$ and, therefore,
$$
\int_S^T\|\Dt u(t)\|^2_{L^2}\,dt\le \lim_{n\to\infty} \int_S^T\|\Dt u_{n}(t)\|^2_{L^2}\,dt.
$$
Taking into account that $g_n\to g$ strongly in $L^2_{loc}(\R_+,L^2)$, it only remains to verify that
 \begin{equation}\label{3.st-c}
 (u_n,\Nx)u_n\to (u,\Nx)u\ \ \ \text{strongly in}\ \ \ L^2_{loc}(\R_+,L^2).
 \end{equation}
 Indeed, due to \eqref{2.imp}, the functions $u_n$ are uniformly bounded in the space $L^2_{loc}(\R_+,W^{1,3q})$ and their time derivatives are uniformly bounded in the space  $L^2_{loc}(\R_+,L^2)$ (due to Theorem \ref{Th2.main}). Therefore, without loss of generality, we have the strong convergence of $u_n$ in $L^2_{loc}(\R_+,W^{1-\eb,3q})$ for all $\eb>0$. Since $W^{1-\eb,3q}\subset L^\infty$ if $\eb>0$ is small enough, we have verified that $u_n\to u$ strongly in $L^2_{loc}(\R_+,L^\infty)$. Using also that $u_n$ is bounded in $L^{8/3}_{loc}(\R_+,L^\infty)$ (again by \eqref{2.imp}), we conclude that
 \begin{equation}\label{3.r1}
 u_n\to u,\ \ \ \text{strongly in}\ \ \ L^{r_1}_{loc}(\R_+,L^\infty)
 \end{equation}
 for all $2<r_1<\frac83$. We now turn to the sequence $\Nx u_n$. According to estimate \eqref{2.en} and the compactness lemma, we conclude that $\Nx u_n\to\Nx u$ strongly in $L^q_{loc}(\R_+,W^{1-\eb,q})$ for all $\eb>0$. Combining this result with the uniform boundedness of  $\Nx u_n$ in $L^2_{loc}(\R_+,L^{3q})$ (due to \eqref{2.imp}) and using that $3q>2$, we end up with the strong convergence
 \begin{equation}\label{3.r2}
 \Nx u_n\to \Nx u,\ \ \ \text{strongly in}\ \ \ L^{r_2}_{loc}(\R_+,L^2)
 \end{equation}
  for all $1<r_2<2$. Moreover, combining this result with the uniform boundedness of the sequence $\Nx u_n$ in $L^\infty(\R_+,L^2)$, we see that the convergence \eqref{3.r2} holds for all $1\le r_2<\infty$. Fixing finally the exponents $r_1$ and $r_2$ in such a way that $\frac1{r_1}+\frac1{r_2}=2$ and using the convergences \eqref{3.r1} and \eqref{3.r2}, we end up with the desired convergence \eqref{3.st-c}. This convergence allows us to pass to the limit $n\to\infty$ in the energy equality \eqref{3.a-enn} and get the desired inequality \eqref{3.e-in}. Thus, the theorem is proved.
\end{proof}
We conclude this section by verifying an interesting fact about the solution $u(t)$ of \eqref{2.BF}, which is also based on the convexity arguments and which is important for what follows in the next section.
\begin{corollary}\label{Cor3.cont} Let the assumptions of Theorem \ref{Th2.main} hold. Then the solution $u(t)$ of the Brinkmann-Forchheimer equation is continuous in time in the strong topology of $\Phi$.
\end{corollary}
\begin{proof} This fact is a standard corollary of the uniform convexity of the function $u\to E(u)$ as the function from $\Phi$ to $\R$ (without loss of generality, we assume that $f'(u)\ge\kappa|u|^{r-1}$ which gives us the desired uniform convexity) and the proved energy equality. Namely, let $t_n\to t$, $t_n\ge0$, be an arbitrary sequence of times. Then, due to the weak continuity of $u(t)$ and the energy equality, we have
\begin{equation}\label{3.good-c}
E(u(t))=\lim_{n\to\infty}E(u(t_n))  \ \ \text{and} \ \ u(t_n)\rightharpoondown u(t) \ \text{in}\ \Phi.
\end{equation}
Since $u\to E(u)$ is uniformly convex, the above convergence imply that $u(t_n)\to u(t)$ strongly in $\Phi$. Although this fact seems to be well-known, see e.g. \cite{Zal83,ZhiP14}, for the convenience of the reader, we present its proof below.
\par
For the quadratic part of the functional $E(u)$, we will use the following obvious identity:
$$
\|\Nx v\|^2_{L^2}-\|\Nx u\|^2_{L^2}=2(\Nx u,\Nx v-\Nx u)+\|\Nx u-\Nx v\|^2_{L^2}.
$$
The analogue of this identity for the nonlinearity $F$ follows from the Taylor expansions near the point $u$, namely,
$$
F(v)-F(u)=f(u).(v-u)+\frac12 \int_0^1(1-s)f'(sv+(1-s)u)\,ds(v-u).(v-u).
$$
It is not difficult to verify that
$$
\int_0^1|su+(1-s)v|^{r-1}\,ds\ge\alpha(|u|^{r-1}+|v|^{r-1})\ge\alpha_1|u-v|^{r-1}
$$
for some positive $\alpha$ and $\alpha_1$. Therefore, putting $u=u(x)$, $v=v(x)$ in the above Taylor's expansion of the nonlinearity $F$  and integrating it with respect to $x$, we get  the inequality
$$
(F(v),1)-(F(u),1)\ge (f(u),v-u)+\frac{\kappa\alpha_1}2\|u-v\|_{L^{r+1}}^{r+1},
$$
which holds for every $u,v\in L^{r+1}(\Omega)$. Combining this inequality with the identity for the quadratic part of $E_u$ mentioned above, we end up with the key inequality
\begin{multline}\label{3.uc}
E(v)-E(u)\ge (\Nx u,\Nx v-\Nx u)+(f(u),v-u)+\\+\frac12\|\Nx u-\Nx v\|^2_{L^2}+\frac{\kappa\alpha_2}2\|u-v\|_{L^{r+1}}^{r+1}\ge\\\ge (\Nx u,\Nx v-\Nx u)+(f(u),v-u)+\alpha_3\|u-v\|^2_\Phi,
\end{multline}
where $\alpha_3$ is a positive constant. This inequality gives us the uniform convexity of $E(u)$ as well as the desired continuity of $u(t)$. Indeed, let us take $u=u(t)$ and $v=u(t_n)$ in this inequality. Then, we have
\begin{multline}\label{3.w-s}
\|u(t_n)-u(t)\|_\Phi^2\le\alpha_3^{-1}\(E(u(t_n))-E(u(t))-\right.\\-\left.(\Nx u(t),\Nx u(t_n)-\Nx u(t))-(f(u(t)),u(t_n)-u(t))\).
\end{multline}
It only remains to note that the convergence \eqref{3.good-c} implies that the right hand side of \eqref{3.w-s} tends to zero as $n\to\infty$, therefore, $u(t_n)\to u(t)$ strongly in $\Phi$. This finishes the proof of the corollary.
\end{proof}

\section{Attractors}\label{s4}

In this section, we will study the long-time behaviour of solutions of problem \eqref{2.BF}. Since the considered equation depends explicitly on time, we will construct a uniform attractor for the cocycle associated with \eqref{2.BF}. First of all, in order to get the dissipativity, we assume that the right-hand side $g(t)$ is in a sense bounded as $t\to\infty$, namely, we assume that
\begin{equation}\label{4.g-bound}
\|g\|_{L^2_b}:=\sup_{h\in\R}\|g\|_{L^2((h,h+1)\times\Omega)}<\infty.
\end{equation}
Following the general theory, we introduce the hull $\Cal H(g)$ via
\begin{equation}\label{4.hul}
\Cal H(g):=\bigg[T(h)g,\ h\in\R\bigg]_{L^{2,w}_{loc}(\R,L^2(\Omega))},
\end{equation}
where $(T(h)g):=g(t+h)$ is a group of temporal shifts and $L^{2,w}_{loc}(\R,L^2(\Omega))$ means the space $L^2_{loc}(\R,L^2(\Omega))$ endowed with the weak topology and $[\cdot]_W$ stands for the closure in the space $W$. Then, due to assumption \eqref{4.g-bound} and the Banach-Alaoglu theorem, the hull $\Cal H(g)$ is compact in $L^{2,w}_{loc}(\R,L^2(\Omega))$ and, obviously, is shift invariant
\begin{equation}\label{4.inv}
T(h)\Cal H(g)=\Cal H(g),\ \ h\in\R.
\end{equation}
From now on we endow the hull $\Cal H(g)$ with the weak topology of the space $L^{2}_{loc}(\R,L^2(\Omega))$, so $\Cal H(g)$ is a compact metrizable topological space.
In order to construct a cocycle associated with problem \eqref{2.BF}, we will consider a family of similar problems with all right-hand sides $\xi(t)$ belonging to the hull $\Cal H(g)$ of the initial right-hand side $g(t)$, namely,
\begin{equation}\label{4.BF}
\Dt u+(u,\Nx)u+\Nx p+f(u)=\Dx u+\xi(t),\  \divv u=0,\ u\big|_{t=0}=u_0,
\end{equation}
where $u_0\in\Phi$ and $\xi\in\Cal H(g)$. Let $\Cal S_\xi(t):\Phi\to\Phi$ be the solution operator of this problem, i.e.
$$
\Cal S_\xi(t)u_0:=u(t),
$$
where $u(t)$ is a unique solution of problem \eqref{4.BF}, which is well-defined due to Theorem \ref{Th2.main}. Then, the operators $\Cal S_\xi(t)$ generate a cocycle in the phase space $\Phi$:
\begin{equation}\label{4.co}
\Cal S_\xi(t_1+t_2)=\Cal S_{T(t_2)\xi}(t_1)\circ\Cal S_\xi(t_2),\ \ \xi\in\Cal H(g),\ t_1,t_2\ge0,
\end{equation}
which is a natural generalization of a solution semigroup to the non-auto\-no\-mo\-us case (see \cite{BV92,CLR12,ChVi02,KlR11,tem,Zel22} and references therein) and which is our main object to study in this section. Recall also that the cocycle property \eqref{4.co} allows us to reduce the considered cocycle to a semigroup, acting in the extended phase space $\Bbb F:=\Phi\times\Cal H(g)$ via
\begin{equation}\label{4.sem}
\Bbb S(t):\Bbb F\to\Bbb F,\ \ \Bbb S(t)(u_0,\xi):=(\Cal S_\xi(t)u_0,T(t)\xi),
\end{equation}
so we may reduce the study of the non-autonomous dynamical system generated by cocycle $\Cal S_\xi(t)$ to the autonomous dynamical system acting on the extended phase space $\Bbb F$. This naturally leads to the uniform attractor. We recall below the basic concepts of the attractor theory adapted to our case.
\begin{definition} A set $B\subset \Phi$ is bounded if $\sup_{u_0\in B}\|u_0\|_{\Phi}<\infty$. A set $\Bbb B\subset\Bbb F$ is bounded if $\Pi_1\Bbb B$ is bounded in $\Phi$. Here and below $\Pi_1:\Bbb F\to\Phi$ is the projector to the first component of the Cartesian product $\Phi\times\Cal H(g)$.
\par
A set $\Cal B$ is a (uniformly) absorbing set for the cocycle $\Cal S_\xi(t)$ if, for any bounded set $B\subset\Phi$, there is a time moment $T=T(B)$ such that
$$
\Cal S_\xi(t)B\subset\Cal B,\ \ \forall\xi\in\Cal H(g),\ \forall t\ge T.
$$
Analogously, a set $\Cal B$ is a uniformly attracting set for the cocycle $\Cal S_\xi(t)$ in $\Phi$ endowed with a suitable topology if, for any bounded set $B\subset\Phi$ and any neighbourhood $\Cal O(\Cal B)$, there exists $T=T(B,\Cal O)$ such that
$$
\Cal S_\xi(t)B\subset\Cal O(\Cal B),\ \ \forall\xi\in\Cal H(g),\ \forall t\ge T.
$$
A set $\Cal A_{un}$ is a uniform attractor for the cocycle $\Cal S_\xi(t)$ if
\par
1) The set $\Cal A_{un}$ is compact in $\Phi$;
\par
2) The set $\Cal A_{un}$  is an attracting set for the cocycle $\Cal S_\xi(t)$;
\par
3) The set $\Cal A_{un}$ is a minimal (by inclusion) set which satisfies properties 1) and 2).
\par
We will consider two choices of the topology in $\Phi$, namely, weak and strong topologies. The corresponding attractors $\Cal A_{un}^w$ and $\Cal A_{un}^s$ will be referred as weak and strong uniform attractors respectively.
\end{definition}
In order to describe the structure of the uniform attractor, we need one more standard definition.

\begin{definition} A function $u:\R\to\Phi$ is a complete bounded trajectory of the cocycle $\Cal S_\xi(t)$ which corresponds to the symbol $\xi\in\Cal H(g)$ if $\sup_{t\in\R}\|u(t)\|_\Phi<\infty$ and
\begin{equation}\label{4.sol}
u(t+h)=\Cal S_{T(t)\xi}(h)u(t),\ \ t\in\R,\ \ h\in\R_+.
\end{equation}
The set of all bounded trajectories of the cocycle $\Cal S_\xi(t)$ which correspond to the symbol $\xi\in\Cal H(g)$ is denoted by $\Cal K_\xi$. It is not difficult to see that a bounded $\Phi$-valued function $u(t)$, $t\in\R$, belongs to $\Cal K_\xi$ if and only if it solves equation \eqref{4.BF} for all $t\in\R$, so the set $\Cal K_\xi\subset C(\R,\Phi)$ is nothing else than the set of all complete bounded solutions of problem \eqref{4.BF} which correspond to the right-hand side $\xi\in\Cal H(g)$.
\end{definition}
We will use the following standard criterion to verify the existence of a uniform attractor.
\begin{proposition}\label{Prop4.crit} Let $\Cal S_\xi(t):\Phi\to\Phi$, $\xi\in\Cal H(g)$ be a cocycle in the phase space $\Phi$ and let $\Cal B$ be a bounded uniformly absorbing set for $\Cal S_\xi(t)$. Then, this cocycle possesses a weak uniform attractor $\Cal A_{un}^w$. If, in addition, the map $(u_0,\xi)\to\Cal S_\xi(t)u_0$ is continuous in a weak topology for any fixed $t\ge0$, then the attractor $\Cal A_{un}^w$ is generated by all complete bounded trajectories of the considered cocycle, namely,
\begin{equation}\label{4.str}
\Cal A_{un}^w=\cup_{\xi\in\Cal H(g)}\Cal K_\xi\big|_{t=0}.
\end{equation}
Finally, if the uniform absorbing set is compact in the strong topology of $\Phi$, then the strong uniform attractor $\Cal A_{un}^s$ exists and coincides with the weak uniform attractor $\Cal A_{un}^w$:
\begin{equation}
\Cal A_{un}:=\Cal A_{un}^s=\Cal A_{un}^w.
\end{equation}
\end{proposition}
The proof of this proposition in a more general setting can be found in \cite{ChVi02}, see also \cite{Zel22}. We mention here that since $\Phi$ is a reflexive Banach space, bounded sets in it are precompact in a weak topology, so if we are given a bounded absorbing/attracting set, its closed convex hull will be a compact absorbing/attracting set, so the standard asymptotic compactness condition is satisfied. We also mention that the uniform attractor $\Cal A_{un}$ is related with the global attractor $\Bbb A$ of the extended semigroup $\Bbb S(t)$ via
$$
\Cal A_{un}=\Pi_1\Bbb A.
$$
We are now ready to study the cocycle generated by the Navier-Stokes-Brinkmann-Forchheimer equation \eqref{4.BF}. We start with the case of the weak uniform attractor.

\begin{theorem}\label{Th4.weak} Let the assumptions of Theorem \ref{Th2.main} hold and let, in addition, assumption \eqref{4.g-bound} be satisfied. Then the cocycle $\Cal S_\xi(t):\Phi\to\Phi$ possesses a weak uniform attractor $\Cal A_{un}^w$ in the phase space $\Phi$ and this attractor is generated by all complete bounded trajectories, i.e. the representation formula \eqref{4.str} holds.
\end{theorem}
\begin{proof}The statement of the theorem is an almost immediate corollary of Theorem \ref{Th2.main}. Indeed, since
$$
\|\xi\|_{L^2_b(\R,L^2(\Omega))}\le \|g\|_{L^2_b(\R,L^2(\Omega))},\ \ \forall\xi\in\Cal H(g),
$$
estimate \eqref{2.main} implies that
\begin{equation}\label{4.dis}
\|\Cal S_\xi(t)u_0\|_\Phi\le Q(\|u_0\|_\Phi)e^{-\beta t}+Q(\|g\|_{L^2_b})
\end{equation}
for some monotone increasing function $Q$ and a positive constant $\beta$ which are independent of $u_0\in\Phi$ and $t\in\R_+$. Estimate \eqref{4.dis} guarantees that the set
\begin{equation}\label{4.abs}
\Cal B:=\{u_0\in\Phi,\ \|u_0\|_\Phi\le 2Q(\|g\|_{L^2_b})\}
\end{equation}
is a bounded uniformly attracting set for the cocycle $\Cal S_\xi(t)$. Therefore, the existence of a weak uniform attractor $\Cal A_{un}^w$ follows from Proposition \ref{Prop4.crit}. The weak continuity of the maps $(u_0,\xi)\to\Cal S_\xi(t)u_0$ for every fixed $t$ can be checked in a standard way (we leave the rigorous proof of this fact to the reader). This gives the representation formula \eqref{4.str} and finishes the proof of the theorem.
\end{proof}
\begin{remark} Recall that the cocycle $\Cal S_\xi(t)$ generates a family of dynamical processes $U_\xi(t,\tau):\Phi\to\Phi$, $t\ge\tau\in\R$, $\xi\in\Cal H(g)$, via
$$
U_\xi(t,\tau):=\Cal S_{T(\tau)\xi}(t-\tau).
$$
Then the cocycle property \eqref{4.co} transforms to
$$
U_\xi(t,\tau)=U_\xi(t,s)\circ U_\xi(s,\tau),\ \ t\ge s\ge\tau\in\R,\ \xi\in\Cal H(g).
$$
Moreover, the uniform attractor $\Cal A_{un}^w$ can be defined using the dynamical process $U_g(t,\tau)$ which corresponds to the initial external force $g$ only (without introducing the hull $\Cal H(g)$ and the cocycle $\Cal S_\xi(t)$). Namely, the set $\Cal A_{un}^w$ is  a weak uniform attractor for the process $U_g(t,\tau)$ if
\par
1) $\Cal A_{un}^w$ is a compact set in $\Phi$ endowed with the weak topology;
\par
2) It possesses a uniform attracting property, i.e. for every bounded set $B$ and every neighbourhood $\Cal O(\Cal A_{un})$ (in a weak topology of $\Phi$), there exists $T=T(\Cal O,B)$ such that
$$
U_g(t,\tau)B\subset \Cal O(\Cal A_{un}^w),\ \ \text{ if} t-\tau\ge T;
$$
\par
3) $\Cal A_{un}^w$ is a minimal (by inclusion) set which satisfies properties 1) and 2),
see \cite{ChVi02,Zel22} for the proof of the equivalence of these definitions. However, if we want to present the attractor as a union of complete bounded trajectories, we need to introduce the hull $\Cal H(g)$ and the cocycle $\Cal S_\xi(t)$. For this reason, we prefer to introduce the cocycle formalism from the very beginning.
\end{remark}
Our next aim is to verify that, under some natural extra assumptions, the constructed weak uniform attractor $\Cal A_{un}^w$ is actually a strong one. We first note that, if the right-hand side $g\in L^2_b(\R,L^2(\Omega))$ only, we cannot expect the existence of a {\it strong} uniform attractor in $H^1(\Omega)$, the corresponding counterexamples can be constructed even on the level of a linear Stokes equation, see \cite{Z15}, so some extra assumptions on $g$ are really necessary. Following the general theory developed in \cite{ChVi02}, see also references therein, the straightforward  assumption which can be posed  is the assumption that $g$ is {\it translation compact} in $L^2_b(\R,L^2(\Omega))$ (i.e. that the hull $\Cal H(g)$ is compact in the {\it strong} topology of $L^2_{loc}(\R,L^2(\Omega))$). This assumption covers the cases when $g$ is periodic, quasi or almost periodic in time or $g(t)$ is a heteroclinic orbit between two stationary right-hand sides or $g(t)$ possesses some extra regularity in both space and time, see \cite{ChVi02} for more examples of translation-compact external forces. However, as it was pointed out in \cite{Lu06,LWZ05,Z15}, the translation compactness assumption can be essentially relaxed. For instance, the extra regularity of $g(t)$ only in space (e.g. $g\in L^2_b(\R,H^1(\Omega))$) or in time (e.g. $g\in H^1_b(\R,L^2(\Omega))$) is often enough to get the strong uniform attractor. We introduce below the most general (to the best of our knowledge) class of external forces, for which we may expect the existence of a strong uniform attractor (at least in the case of parabolic PDEs) and verify that this is indeed true for the case of our problem \eqref{4.BF}.

\begin{definition}\label{Def4.w-n} A function $g\in L^2_{b}(\R,L^2(\Omega))$ is weakly normal if for every $\eb>0$ there exists a finite-dimensional subspace $H_\eb\subset L^2(\Omega)$ and a splitting
$$
g=\tilde g_\eb+\bar g_\eb
$$
such that $\tilde g_\eb\in L^2_b(\R,H_\eb)$ and the function $\bar g_\eb$ satisfies the condition
\begin{equation}\label{4.weak}
\limsup_{s\to0}\sup_{t\in\R}\int_t^{t+s}\|\bar g_\eb(\tau)\|^2_{L^2}\,d\tau\le \eb.
\end{equation}
Important for us that the finite-dimensional  spaces $H_\eb$ can be chosen in such a way that $H_\eb\subset C_0^\infty(\Omega)$, see \cite{Z15} for more details. Recall also that the function $g\in L^2_b(\R,L^2(\Omega))$ is called normal if
\begin{equation}\label{4.norm}
\lim_{s\to0}\sup_{t\in\R}\int_t^{t+s}\|g(\tau)\|^2_{L^2}\,d\tau=0.
\end{equation}
It is worth mentioning as well that if $g$ is weakly normal, then any $\xi\in\Cal H(g)$ is also weakly normal, namely, for any $\xi\in\Cal H(g)$, there exist $\tilde\xi_\eb\in\Cal H(\tilde g_\eb)$ and $\bar \xi_\eb\in\Cal H(\bar g_\eb)$ such that $\xi=\tilde\xi_\eb+\bar\xi_\eb$ and the functions $\bar\xi_\eb$ satisfy \eqref{4.weak} uniformly with respect to $\xi\in\Cal H(g)$.
\end{definition}

The next theorem can be considered as the main result of this section.

\begin{theorem}\label{Th4.main} Let the assumptions of Theorem \ref{Th4.weak} hold and let also the right-hand side $g\in L^2_b(\R,L^2(\Omega))$ be weakly normal. Then the cocycle $\Cal S_\xi(t)$ associated with problem \eqref{4.BF} possesses a strong uniform attractor $\Cal A_{un}^s$ which coincides with the weak uniform attractor constructed above.
\end{theorem}
\begin{proof} According to Proposition \ref{Prop4.crit}, we only need to find a (pre)compact uniformly absorbing set for the cocycle $\Cal S_\xi(t)$. We claim that the set
\begin{equation}\label{4.comp}
\Cal B_1:=\cup_{\xi\in\Cal H(g)}\Cal S_\xi(1)\Cal B,
\end{equation}
where $\Cal B$ is defined via \eqref{4.abs} is a desired absorbing set. Indeed, due to estimate \eqref{4.dis} this set is absorbing and bounded, so we only need to check its compactness. Let $v_n\in\Cal B_1$ is an arbitrary sequence. Then there exists a sequence of the initial data $u_0^n\in\Cal B$ and a sequence of right-hand sides $\xi_n\in\Cal H(g)$ such that $v_n=u_n(1)$, where $u_n(t):=\Cal S_{\xi_n}u_0^n$ is the corresponding sequence of solutions of problem \eqref{4.BF}. Since $\Cal B$ and $\Cal H(g)$ are compact in a weak topology, we may assume without loss of generality that
$$
u_0^n\rightharpoondown u_0 \  \text{ and }\ \xi_n\rightharpoondown\xi.
$$
Moreover, since the maps $(u_0,\xi)\to\Cal S_\xi(t)u_0$ are continuous in a weak topology for every fixed $t\ge0$, we conclude that $u_n(t)\rightharpoondown u(t)$ in $\Phi$ for every fixed $t$, where
$u(t):=\Cal S_\xi(t)u_0$. In particular, $v_n\rightharpoondown u(1)$. Therefore, we only need to verify that this convergent is actually strong.
\par
 Arguing as in the proof of Corollary \ref{Cor3.cont}, we see that we only need to verify that $E(u_n(1))\to E(u(1))$. In turn, in order to verify this convergence, we will use the energy identity \eqref{3.en} together with the trick suggested in \cite{Z15}.  In order to avoid technicalities, we first give the proof for the particular case where $g$ is normal and after that indicate briefly the changes which should be made to cover the general case where $g$ is only weakly normal.
 \par
 Namely, we multiply equation \eqref{4.BF} for solutions $u_n$ by $Nu_n$, where $N$ is a big positive constant, integrate over $x\in\Omega$ and take a sum with \eqref{3.en} for $u_n$. This gives the following identity:
 \begin{multline}
 \frac12\frac d{dt}E(u_n(t))\!+\!2NE(u_n(t))\!+\!\|\Dt u(t)\|^2_{L^2}\!+\!N(f(u_n).u_n-2F(u_n),1)\\=
 (\xi_n,\Dt u_n+Nu_n)-((u_n,\Nx)u_n,\Dt u_n).
 \end{multline}
 Multiplying this identity by $t$ and integrating over $t\in[0,1]$, we arrive at
 \begin{multline}\label{4.huge}
 E(u_n(1))+2\int_0^1te^{-4N(1-t)}\|\Dt u_n(t)\|^2_{L^2}\,dt+\\+2N\int_0^1te^{-4N(1-t)}(f(u_n(t)).u_n(t)-2F(u_n(t)),1)\,dt=\\=
 \int_0^1e^{-4N(1-t)}E(u_n(t))\,dt+2\int_0^1te^{-4N(1-t)}((u_n(t),\Nx)u_n(t),\Dt u_n(t))\,dt+\\+2N\int_0^1e^{-4N(1-t)}t(\xi_n(t),u_n(t))\,dt+2\int_0^1te^{-4N(1-t)}(\xi_n(t),\Dt u_n(t))\,dt.
 \end{multline}
 We want to pass to the limit $n\to\infty$ in this identity. To this end, we note that by convexity arguments
 $$
 2\int_0^1te^{-4N(1-t)}\|\Dt u(t)\|^2_{L^2}\,dt\le 2\liminf_{n\to\infty}\int_0^1te^{-4N(1-t)}\|\Dt u_n(t)\|^2_{L^2}\,dt.
 $$
 To pass to the limit in the other terms, we recall that, due to estimates \eqref{2.en} and \eqref{2.imp}, the sequence $u_n$ is bounded in the space
 $$
 L^{8/3}(0,1;L^\infty(\Omega)\cap L^2(0,1;W^{1.3q}(\Omega))\cap H^1(0,1;L^2(\Omega))\cap L^\infty(0,1;L^{r+1}(\Omega))
 $$
 and therefore, since this space is compactly embedded to $L^2(0,1;W^{1,2}(\Omega))\cap L^{r+1}(0,1;L^{r+1}(\Omega))$, we have the strong convergence $u_n\to u$ in this space. This allows us to pass to the limit in the last term in the left-hand side of \eqref{4.huge} as well as in the first term in the right-hand side there. Moreover, arguing analogously to the proof of Theorem \ref{Th3.main} (see \eqref{3.st-c}), we prove that $(u_n,\Nx)u_n\to (u,\Nx)u$ strongly in $L^2(0,1;L^2(\Omega))$ and this allows us to pass to the limit in the second term in the right-hand side of \eqref{4.huge}. The third term is obvious since $u_n\to u$ strongly in $L^2(0,1;L^2(\Omega))$ and the only problem is the last term in the right-hand side. In contrast to the other terms, we have here only weak convergence $\xi_n\to\xi$ and $\Dt u_n\to\Dt u$ in $L^2(0,1;L^2(\Omega))$ and cannot use the convexity arguments. Instead, we prove that this term is small when $N$ is large uniformly with respect to $n$ and this will be enough for our purposes.
 \begin{lemma} Under the above assumptions, the following convergence holds:
 \begin{equation}
 \lim_{N\to\infty}\sup_{n\in\Bbb N}\int_0^1te^{-4N(1-t)}|(\xi_n(t),\Dt u_n(t))|\,dt=0.
 \end{equation}
 \end{lemma}
\begin{proof}[Proof of the lemma] Indeed, since $\Dt u_n$ are uniformly bounded in the space  $L^2(0,1;L^2(\Omega))$, due to Cauchy-Schwarz inequality, we have
$$
\int_0^1te^{-4N(1-t)}|(\xi_n(t),\Dt u_n(t))|\,dt\le C\(\int_0^1te^{-4N(1-t)}\|\xi_n(t)\|^2_{L^2}\,dt\)^{1/2}.
$$
It only remains to recall that assumption \eqref{4.norm} implies that
\begin{equation}\label{4.n-est}
\lim_{N\to\infty}\sup_{t\in\R}\sup_{\xi\in\Cal H(g)}\int_{-\infty}^te^{-N(t-s)}\|\xi(s)\|^2_{L^2}\,ds=0,
\end{equation}
see \cite{Z15,Zel22} for the details. This finishes the proof of the lemma.
\end{proof}
We now ready to finish the proof of the theorem for the case of the theorem for the case where $g$ is normal. Let $\eb>0$ be arbitrary. Then, due to the lemma, we may fix $N=N(\eb)$ in such a way that the last term in the right-hand side of \eqref{4.huge} will be less than $\eb$. Passing to the limit $n\to\infty$ in \eqref{4.huge} then gives
 \begin{multline}\label{4.huge1}
 \limsup_{n\to\infty}E(u_n(1))+2\int_0^1te^{-4N(1-t)}\|\Dt u(t)\|^2_{L^2}\,dt+\\+2N\int_0^1te^{-4N(1-t)}(f(u(t)).u(t)-2F(u(t)),1)\,dt\le\\\le
 \int_0^1e^{-4N(1-t)}E(u(t))\,dt+2\int_0^1te^{-4N(1-t)}((u(t),\Nx)u(t),\Dt u(t))\,dt+\\+2N\int_0^1e^{-4N(1-t)}t(\xi(t),u(t))\,dt+\eb.
 \end{multline}
 Comparing this inequality with the identity \eqref{4.huge} for the limit solution $u(t)$ and using the lemma again, we end up with the inequality
 $$
 \limsup_{n\to\infty}E(u_n(1))\le E(u(1))+2\eb.
 $$
 Since $\eb>0$ is arbitrary, passing to the limit $\eb\to0$ gives us the inequality
 \begin{equation}\label{4.nice}
 \limsup_{n\to\infty}E(u_n(1))\le E(u(1)).
\end{equation}
 The opposite inequality follows from the weak convergence $u_n(1)\to u(1)$ in $\Phi$ and the convexity arguments. Thus, we have proved the convergence $E(u_n(1))\to E(u(1))$ which gives us the desired strong convergence $u_n(1)\to u(1)$ in $\Phi$ and finishes the proof of the theorem in the case of normal external forces $g$.
 \par
 Let us assume now that $g$ is weakly normal. In this case, the external forces $\xi_n$ can be split in two parts $\tilde \xi_{n,\eb}\in L^2_b(\R,H_\eb)$, where $\eb>0$ is arbitrary and $H_\eb$ is smooth and finite-dimensional and $\bar\xi_{n,\eb}$ which satisfy \eqref{4.weak} uniformly with respect to $n$. The analogue of  \eqref{4.n-est} now reads
\begin{equation}\label{4.wn-est}
\limsup_{N\to\infty}\sup_{t\in\R}\sup_{\xi\in\Cal H(g)}\int_{-\infty }^te^{-N(t-s)}\|\bar\xi_{\eb}(s)\|^2_{L^2}\,ds\le C\eb,
\end{equation}
where $C$ is independent of $\xi\in\Cal H(g)$, see \cite{Z15}. Therefore, as in the case of normal external forces, the term containing $(\xi_{n,\eb}(s),\Dt u_n(s))$ can be made arbitrarily small by choosing $N$ and $\eb$ big and small enough respectively. In contrast to this, the term containing $(\tilde\xi_{n,\eb}(s),\Dt u_n(s))$ cannot in general be made small and we should treat it in a different way. Namely, we introduce a corrector $v_{n,\eb}(t)$ as a solution of the following linear Stokes problem
\begin{equation}\label{4.au}
\Dt v_{n,\eb}+\Nx q_n=\Dx v_{n,\eb}+\tilde \xi_{n,\eb}(t),\ \ \divv v_{n,\eb}=0,\ \ v_{n,\eb}\big|_{t=0}=0.
\end{equation}
Then, since $\tilde \xi_{n,\eb}$ are uniformly with respect to $n$ smooth, using the anisotropic  $L^2(0,1;L^s(\Omega))$-maximal regularity estimate for the solutions of the Stokes equation, for every finite $s$, we have the uniform estimate
\begin{equation}\label{4.v-good}
\|v_{n,\eb}\|_{C(0,1; W^{1+\delta,2}(\Omega))}+\|\Dt v_{n,\eb}\|_{L^2(0,1;L^s(\Omega))}\le C_{\eb,s}.
\end{equation}
This estimate shows that, without loss of generality, we may assume that
\begin{equation}\label{4.conv}
v_{n,\eb} \to v_{\eb}\ \ \text{in $C(0,1;\Phi)$ and } \Dt v_{n,\eb}\rightharpoondown \Dt v_{\eb} \ \text{in $L^2(0,1; L^s(\Omega))$}.
\end{equation}
Let $w_n:=u_n-v_{n,\eb}$. Then this function solves
\begin{equation}\label{4.BFw}
\Dt w_n+(u_n,\Nx)u_n+\Nx p_n+f(u_n)=\Dx w_n+\bar\xi_{n,\eb},\ \ \divv w_n=0.
\end{equation}
We write the analogue of \eqref{4.huge} for this equation by multiplying it by $t(\Dt w_n+Nw_n)$ and integrating over $x\in\Omega$ and time. Due to the presence of the corrector $v_{n,\eb}$, we will have several extra terms in this equality, namely, the energy $E(u_n(t))$ will be replaced by
$$
\tilde E(u_n(t)):=\frac12\|\Nx w_n(t)\|^2_{L^2}+(F(u_n(t),1).
$$
From the nonlinearity $f$, we will have an extra term which contains
$$
(f(u_n(t)),\Dt v_{n,\eb}(t)+N v_{n,\eb}(t_)).
$$
 This term is not dangerous since from \eqref{4.conv} we have the weak convergence of $\Dt v_{n,\eb}+N v_{n,\eb}$ in $L^2(0,1;L^s(\Omega))$ for all $s<\infty$ and, arguing as above, we may check that $f(u_n)\to f(u)$ strongly in $L^2(0,1;L^{1+\delta}(\Omega))$ for sufficiently small positive $\delta$. We will also have the term containing $(u_n,\Nx)u_n,\Dt v_{n,\eb}+Nv_{n,\eb})$ which is not dangerous as well since we have the strong convergence of $(u_n,\Nx) u_n$ in $L^2(0,1;L^2(\Omega))$.
\par
The main advantage of this modified energy identity is that the bad term containing $(\tilde \xi_{n,\eb},\Dt u_n)$ disappears and we may argue as in the case of normal external forces and verify that
\begin{equation}\label{4.intr}
\limsup_{n\to\infty}\tilde E(u_n(1))\le \tilde E(u(1))+C\eb.
\end{equation}
Moreover, since $v_{n,\eb}(1)\to v_\eb(1)$ strongly in $\Phi$, we may replace
$\tilde E(u_n(1))$ by $E(u_n(1))$ in \eqref{4.intr}. Finally, since $E(u_n(1))$ is independent of $\eb$, we may pass to the limit $\eb\to\infty$ and end up with inequality \eqref{4.nice} in the case of weakly normal external forces as well. Then, as before, we derive the opposite inequality by the convexity arguments and finally establish that $u_n(1)\to u(1)$ strongly in $\Phi$. This finishes the proof of the theorem.
\end{proof}

\end{document}